\documentclass{siamltex}

\usepackage{hyperref} 
  \hypersetup{hidelinks = true}
\usepackage[caption=true]{subfig}

\usepackage{amsmath}
\usepackage{amssymb}
\usepackage{amsfonts}
\usepackage{dsfont}
\usepackage{bm}
\usepackage{graphicx}
\usepackage{epstopdf}
\usepackage{nicefrac}
\usepackage{floatrow}
\usepackage{url}

\usepackage{geometry}
\usepackage{marginnote}

\usepackage{algpseudocode}
\usepackage{algorithm}

\usepackage{color}

\newcommand{\kron}{\otimes}

\newcommand{\bfd}{{\bm d}}
\newcommand{\bfe}{{\bm e}}

\newcommand{\bfx}{{\bm x}}
\newcommand{\bfr}{{\bm r}}
\newcommand{\bfzero}{{\bm 0}}
\newcommand{\mxA}{{\bm A}}
\newcommand{\mxB}{{\bm B}}
\newcommand{\mxD}{{\bm D}}
\newcommand{\mxI}{{\bm I}}
\newcommand{\mxK}{{\bm K}}
\newcommand{\mxL}{{\bm L}}
\newcommand{\mxT}{{\bm T}}
\newcommand{\mxU}{{\bm U}}
\newcommand{\mxV}{{\bm V}}
\newcommand{\mxW}{{\bm W}}

\newcommand{\mxSigma}{{\bm \Sigma}}
\newcommand{\bo}[1]{{\bm #1}}
\newcommand{\mx}[1]{{\bm #1}}
\newcommand{\sslink}[1]{\hyperref[ssec:#1]{\S\ref*{ssec:#1}}}

\newcommand{\vecop}{\mbox{vec}}

\def\Re{{\cal R}}

\title{SINGULAR VALUE DECOMPOSITION APPROXIMATION VIA KRONECKER SUMMATIONS 
FOR IMAGING APPLICATIONS}

\author{CLARISSA~GARVEY\thanks{Department of Mathematics and Computer Science, 
Emory University, Atlanta, GA, USA. \texttt{(ccgarve@emory.edu, chang.meng@emory.edu, jnagy@emory.edu)}. This work was supported by grant no.~DMS-1522760 from the US National Science Foundation.}
\and CHANG MENG\footnotemark[1]
\and JAMES~G.~NAGY\footnotemark[1]
}

\begin{document}

\maketitle

\begin{abstract}
\noindent
In this paper we propose an approach to approximate a truncated singular value decomposition of a large structured matrix.  By first decomposing the matrix into a sum of Kronecker products, our approach can be used to approximate a large number of singular values and vectors more efficiently than other well known schemes, such as randomized matrix algorithms or iterative algorithms based on Golub-Kahan bidiagonalization. We provide theoretical results and numerical experiments to demonstrate the accuracy of our approximation and show how the approximation can be used to solve large scale ill-posed inverse problems, either as an approximate filtering method, or as a preconditioner to accelerate iterative algorithms.
\end{abstract}



\pagestyle{myheadings} \thispagestyle{plain} \markboth{C.~GARVEY, C.~MENG
AND J.~NAGY}{SVD APPROXIMATION VIA KRONECKER SUMMATION}

\section{Introduction} \label{sec:introduction}
In this paper we are concerned with computing approximations of large scale linear systems that arise from discretization of ill-posed inverse problems in imaging applications. In these applications, the aim is to compute an approximation of a vector $\bfx$ from measured data $\bfd$,
\begin{equation}
\label{eq:DIP}
  \bfd = \mxK \bfx + \bfe,
\end{equation}
where $\mxK$ is an ill-conditioned matrix whose singular values tend to zero with no significant gap to indicate numerical rank, and $\bfe$ represents unknown data measurement errors. Due to the ill-conditioning of $\mxK$ and the presence of noise, regularization is necessary to compute an accurate approximation of $\bfx$ \cite{calvetti2007bayesian, engl2000regularization, hansen2010discrete, mueller2012linear, vogel2002computational}.

A well known approach for regularization is to use a singular value decomposition (SVD) filtering technique. Specifically, if the SVD of $\mxK \in \Re^{N \times N}$ is denoted by
$$
  \mxK = \mxU \mxSigma \mxV^T = \mxU \mbox{diag}\left(\sigma_1, \cdots, \sigma_N\right)\mx{V}^T
$$ 
then an SVD filtered approximate solution of $\bfx$ is given by
\begin{equation}
\label{eq:SVDfilter}
  \bfx_F = \mxV \mxSigma_F^\dagger \mxU^T\bfd
  = \mxV \mbox{diag}\left(\frac{\phi_1}{\sigma_1}, \cdots, \frac{\phi_N}{\sigma_N}\right) \mxU^T\bfd\,,
\end{equation}
where it is assumed that if $\sigma_i = 0$, then $\phi_i/\sigma_i = 0$.
The choice of regularization scheme defines the filter factors $\phi_i$.
For example, in the case of truncated SVD (TSVD),
$$
  \phi_i = \left\{ \begin{array}{ll} 1 & \mbox{ if } \; i \leq k \\ 0 & \mbox{ if } \; i > k\,, \end{array} \right.
$$
where $k$ is a specified truncation index.  Another well known example is Tikhonov filtering where, for a chosen regularization parameter $\alpha$,
$$
  \phi_i = \frac{\sigma_i^2}{\sigma_i^2 + \alpha^2} 
$$
(see \text{\cite{tikhonov1977solutions}}).

In the case of TSVD, suppose we partition the matrices to identify terms above and below the truncation index:
\begin{equation}
\label{eq:FullSVD}
  \mxK = \left[ \begin{array}{cc} \mxU_k & \mxU_0 \end{array} \right]
  \left[ \begin{array}{cc} \mxSigma_k & 0 \\ 0 & \mxSigma_0 \end{array} \right]
  \left[ \begin{array}{c} \mxV_k^T \\[3pt] \mxV_0^T \end{array} \right]
  = \mxU_k\mxSigma_k\mxV_k^T + \mxU_0\mxSigma_0\mxV_0^T\,,
\end{equation}
where $\mxSigma_k = \mbox{diag}(\sigma_1, \sigma_2, \ldots, \sigma_k)$; $\mxU_k$ and $\mxV_k$ are, respectively, the first $k$ columns of $\mxU$ and $\mxV$; and the other submatrices are defined accordingly.  Then we can define the TSVD operator as
\begin{equation}
\label{eq:TSVDoperator}
  \mxK_\text{TSVD} = \mxU_k\mxSigma_k\mxV_k^T\,,
\end{equation}
and the TSVD filtered solution as
\begin{equation}
\label{eq:TSVDSol}
  \bfx_\text{TSVD} = \mxK^{\dagger}_\text{TSVD}\bfd
  = \mxV_k\mxSigma_k^{-1}\mxU_k^T\bfd\,.
\end{equation}

The idea of SVD filtering is motivated by the fact that if we compute the inverse solution 
$$
  \bfx_\text{inv} = \mxK^{-1}\bfd = \mxK^{-1}(\mxK\bfx + \bfe) = \bfx + \mxK^{-1}\bfe\,,
$$
then dividing by the smallest singular values will highly amplify the noise $\bfe$. SVD filtering avoids dividing by these small values.

Although the above description is for square matrices, the ideas easily extend to over and underdetermined matrices. Moreover, the specific choices of parameters, such as the truncation index $k$ for TSVD filtering and the value of $\alpha$ for Tikhonov filtering, depend on the problem and data.  There are computational methods such as generalized cross validation, discrepancy principle, L-curve, etc., to help choose appropriate values; details exist in literature on inverse problems \cite{calvetti2007bayesian, engl2000regularization, hansen2010discrete, mueller2012linear, vogel2002computational}.

A major drawback of SVD filtering is that it has high computational cost for large scale problems, such as those that arise in 2-dimensional and 3-dimensional imaging applications.  For example, in image restoration $\bfd$ is a vector representation of an observed blurred image, $\bfx$ is a vector representation of the corresponding clean image, and $\mxK$ models the blurring operation. If the images that define $\bfd$ and $\bfx$ have $n \times n$ pixels and $n^2 = N$, then $\bfd, \; \bfx \in \Re^N$ and the blurring operator $\mxK \in \Re^{N \times N}$.  Typical image sizes are at least $256 \times 256$ pixels and often larger, making the blur operator $\mxK$ at least $65536 \times 65536$.  Computing an SVD of such a large matrix is, in most cases, prohibitively expensive. However, there are exceptions.

In the case of spatially-invariant blur the operator may be structured in a way that enables cheap computation. The exact structure of the operator depends on the corresponding point-spread function (PSF), which represents the blur incurred on a single point-source of light, and on chosen boundary conditions. For specific boundary conditions, the resulting structure enables cheap computation. For example, in the case of periodic boundary conditions, the SVD can be replaced by an equivalent Fourier based spectral decomposition, and implemented in $O(N\log(N))$ $=$ $O(n^2\log(n))$ floating point operations using fast Fourier transforms (FFTs) \cite{nagy2006deblurring}. Although these methods are fast, their performance degrades for non-periodic images \cite{nagy2006deblurring,kamm1998kronecker}. One of the main motivations of this work is to reasonably accurately, but still cheaply, enable an approximate factorization for a larger variety of structures.

For matrices that do not have directly exploitable structure, the SVD and similar direct factorization methods, such as
rank revealing decompositions \cite{fierro1995accuracy}, are very expensive to compute for large scale problems. It is therefore often necessary to use iterative Krylov subspace methods \cite{engl2000regularization,vogel2002computational} to make the problems computationally tractable. These can be implemented directly on the system $\mxK \bfx = \bfd$ or the least squares problem $\displaystyle \min_{\bfx}\|\mxK\bfx - \bfd\|_2$, and regularization is enforced through early termination of the iterations; this is referred to as iterative regularization. An alternative approach is to use an iterative method on a damped least squares problem, e.g., in the case of Tikhonov regularization
$$
 \min_{\bfx} 
 \left\|
  \left[ \begin{array}{c}
    \mxK \\
    \alpha \mxI
  \end{array}\right]
  \bfx
  -
  \left[ \begin{array}{c}
    \bfd \\
    \bfzero
  \end{array} \right]
 \right\|_2\,.
$$
Such methods can attain accelerated convergence through the use of preconditioning. For well-posed problems, good preconditioners approximate the system matrix (or its inverse) and are cheap to apply; more accurate approximations typically lead to faster convergence.  However, for ill-posed problems we typically do not want to invert small singular values of the matrix (this is most easily seen in the TSVD filter). A good preconditioner should therefore only approximate the matrix corresponding to large singular values of $\mxK$.  This observation is the second motivation of the method presented in this paper.

In this work, we explore an approach to compute approximations of the largest singular values and corresponding singular vectors of a large scale matrix $\mxK$. To do so, we use the Kronecker product operator $\otimes$,
$$
\mxA \kron \mxB
=
\begin{bmatrix}
  a_{11}\mxB & a_{12}\mxB & \hdots & a_{1n}\mxB\\
  a_{21}\mxB & a_{22}\mxB & \hdots & a_{2n}\mxB\\
  \vdots & \vdots & & \vdots \\
  a_{n1}\mxB & a_{n2} \mxB & \cdots & a_{nn}\mxB
\end{bmatrix}.
$$
The basis of the approximation is a decomposition of $\mx{K}$ into a sum of Kronecker products \cite{van1993approximation}. If $\mxK \in \Re^{N \times N}$, then there exist $\mxA_i, \, \mxB_i \in \Re^{n \times n}$ such that 
\begin{equation}\label{equation:krondecomposition}
\mxK = \sum\limits_{i=1}^R \mxA_i \otimes \mxB_i \, .
\end{equation}
$R$, the number of terms needed for this summation to be exact, is known as the Kronecker rank of $\mx{K}$. Section \ref{sec:kron_sum} further details this decomposition.

The decomposition (\ref{equation:krondecomposition}) is the foundation of our TSVD approximation. We detail how to construct the approximated TSVD from the Kronecker summation decomposition in Section \ref{sec:method}. The approximated TSVD can be used directly to approximate SVD filtered solutions, or indirectly as preconditioners for iterative methods. Our algorithm seeks to improve on existing methods using Kronecker summation decompositions to construct approximate TSVDs (see Section~\ref{sec:method} and Section~\ref{ssec:related_kron} for details of prior Kronecker-based work).

There are alternatives to Kronecker product decomposition approaches. For example, the iterative Golub-Kahan bidiagonalization (GKB) method \cite{golub2013matrix,larsen2004propack} can be used to estimate some of the large singular values and corresponding singular vectors (e.g., as implemented in MATLAB's {\tt svds} function). While this approach is efficient if only a small number (e.g., 10) of singular components are required, it is not computationally attractive in our applications because we may need to compute on the order of  1000 singular values and corresponding vectors. Similar constraints hold for randomized algorithms \cite{halko2011finding}. However, as mentioned in Section~\ref{sec:kron_sum}, GKB and randomized methods may be used to decompose $\mxK$ into a sum of Kronecker products (\ref{equation:krondecomposition}), which is the first step in the method proposed in this paper. Further details of these alternative methods are in Section~\ref{sec:related_works}.

One of the primary strengths of the proposed TSVD algorithm is its computational speed, especially relative to the existing alternative methods. Section~\ref{ssec:complexity} contains a derivation of the time complexity of the algorithm. We also tested our method to see how well it runs in practice; this can be found in Section~\ref{sec:experiments}. The method is both fast and accurate enough to be useful in a variety of settings.

\section{Kronecker Sum Decomposition}
\label{sec:kron_sum}

Van Loan and Pitsianis proposed a computational approach for decomposing a general matrix into a sum of Kronecker products (\ref{equation:krondecomposition}) \cite{van1993approximation}. This approach requires taking the singular value decomposition of a rearrangement $\widetilde{\mx{K}}$ of the matrix $\mx{K}$; for details, see \cite{van1993approximation}. The singular vectors of $\widetilde{\mx{K}}$, scaled by the square root of the corresponding singular values, are rearranged into matrices to form the terms $\mx{A}_i$ and $\mx{B}_i$ in (\ref{equation:krondecomposition}).

Our aim is to compute approximations of the largest singular values
and corresponding singular vectors of the large matrix $\mxK$, but we begin the process by first solving the same problem for a different and equally large matrix, $\widetilde{\mx{K}}$. An obvious question is: Why should this save any computational costs? The answer is that often the Kronecker rank of $\mx{K}$ (which is the usual matrix rank of $\widetilde\mxK$) is substantially smaller than the rank of $\mxK$.  A simple example to illustrate this is the discrete 2-dimensional Laplacian matrix on an $n \times n$ grid,
$$
  \mxL
  =
  \left[ \begin{array}{rrrr}
    \mxT & -\mxI & & \\ -\mxI & \mxT & \ddots & \\
    & \ddots & \ddots & -\mxI \\
    & & -\mxI & \mxT
  \end{array} \right]
  \,, \quad
  \mxT
  =
  \left[ \begin{array}{rrrr}
    4 & -1 & & \\
    -1 & 4 & \ddots & \\
    & \ddots & \ddots & -1 \\
    & & -1 & 4
  \end{array} \right]
$$
where rank$(\mxL) = n^2$ but rank$(\widetilde\mxL) = 2$. Although this example is trivial, it indicates that we can expect the Kronecker rank to be significantly less than the matrix rank for certain structured and sparse matrices.  In particular, in some imaging applications the Kronecker rank of $\mx{K} \in \Re^{N \times N}$ is at most $n = \sqrt{N}$ and is often much smaller in practice. Moreover, by exploiting structure of the matrix, the actual computational cost of computing the Kronecker sum decomposition is at most $O(n^3) = O(N^{3/2})$; for details, see \cite{kamm1998kronecker, kamm2000optimal, kilmer2007kronecker, nagy2006kronecker, nagy2003kronecker, perrone2006kronecker}. If the matrix does not have such an exploitable structure, but is sparse, then GKB or randomized methods applied to the matrix $\widetilde\mxK$ are potential alternatives.

Computing the exact decomposition (\ref{equation:krondecomposition}) of $\mx{K}$ is cheap for well-structured matrices, but using an inexact decomposition will lower the cost of later steps in computing an approximate TSVD (this is described in Section~\ref{ssec:complexity}). If there is a large gap in the singular values of $\widetilde\mxK$ between indices $r$ and $r+1$, where $r < R$, then
$$
  \mxK \approx \sum\limits_{i=1}^r \mxA_i \otimes \mxB_i\,,
$$
provides a good estimate of $\mx{K}$.  This is explored in more detail through numerical experiments in Section~\ref{sec:experiments}.

\section{Method} 
\label{sec:method}
In this section we present an algorithm for computing an SVD approximation of $\mxK$ using the Kronecker sum decomposition (\ref{equation:krondecomposition}). From the discussion in the previous section, each of the matrices $\mxA_i$ and $\mxB_i$ correspond to the singular values and vectors of $\widetilde\mxK$. The $i\textsuperscript{th}$ term in the summation corresponds to the $i\textsuperscript{th}$ most significant singular value of $\widetilde\mxK$. We therefore have
$$
  \mxA_1 \kron \mxB_1 = \arg\min_{\mxA,\mxB} \| \mxK - \mxA \kron \mxB\|_F
$$
and, intuitively, $\mx{A}_1$ and $\mx{B}_1$ contain the most information about $\mxK$ of any of the individual pairs of $\mxA_i$ and $\mxB_i$. We therefore treat $\mx{A}_i$ and $\mx{B}_i$ separately from the other terms in the summation.

Computing the SVD of $\mxA_1 \kron \mxB_1$ only requires computing the SVDs of the small matrices $\mxA_1$ and $\mxB_1$. If $\mx{A}_1 = \mx{U}_A \mx{\Sigma}_A \mx{V}_A^T$ and $\mx{B}_1 = \mx{U}_B \mx{\Sigma}_B \mx{V}_B^T$, then by the properties of Kronecker products
\begin{align*}
\mxA_1 \kron \mxB_1 & = \mxU_A \mxSigma_A \mxV_A^T \kron \mxU_B \mxSigma_B \mxV_B^T\\
& = (\mxU_A \kron \mxU_B) (\mxSigma_A \kron \mxSigma_B) (\mxV_A \kron \mxV_B)^T\\
& = \mxU_1 \mxSigma_1 \mxV_1^T.
\end{align*}
Note that $\mxU_1 = \mxU_A \kron \mxU_B, \; \mxSigma_1 = \mxSigma_A \kron \mxSigma_B,$ and $\mxV_1 = \mxV_A \otimes \mxV_B$ are never formed explicitly; maintaining the Kronecker product forms is spatially cheaper and computationally faster.

Having computed $\mx{K}_1 = \mx{U}_1 \mx{\Sigma}_1 \mx{V}_1^T$, we could then consider using the SVD approximation
$$
   \mxK \approx \mxU_1 \mxSigma_1 \mxV_1^T\,,
$$
which is inexpensive to both construct and to apply as either an approximate filtering method or as a preconditioner.  Multiplications with $\mx{U}_1$ and $\mx{V}_1$ (or their transposes) are cheap due to a property of Kronecker products. For example, to multiply the matrix $\mx{U}_1^T$ with $\bo{d}$ as in (\ref{eq:SVDfilter}), we compute
$$
  \mxU^T_1\bfd = (\mxU_A^T \kron \mxU_B^T)\bfd 
    = \vecop(\mxU_B^T\mxD\mxU_A)\,,
  \quad \bfd = \vecop(\mxD)\,,
$$
where $\vecop(\mx{D})$ is the reshaping of $\bo{d}$ into a matrix in column-major order. A similar multiplication works for $\mxV_1$. Multiplication with the diagonal matrix $\mx{\Sigma}_1$ is cheaper still.

The disadvantage of this simple approach is that it uses only the first term in the Kronecker sum decomposition. Kamm and Nagy \cite{kamm1998kronecker, nagy1996decomposition} proposed using more terms via the approximation
$$
  \mxK \approx \mxU_1 \widehat\mxSigma_1 \mxV_1^T
$$
where $\mx{U_1}$ and $\mx{V_1}$ are as described above, and $\widehat\mxSigma_1 = \mbox{diag}(\mxU_1^T\mxK\mxV_1)$.  In this approximation, the singular vectors are fixed to be those coming from the first term in the Kronecker sum decomposition, and $\widehat\mxSigma_1$ is the best diagonal matrix in the sense that it minimizes $\displaystyle \|\mxK - \mxU_1 \widehat\mxSigma \mxV_1^T\|_F$ over all diagonal
matrices $\widehat\mxSigma$.

This baseline method is both computationally efficient and uses more information than just $\mx{A}_1$ and $\mx{B}_1$. However, the singular vectors are constructed using only the first term of the Kronecker sum decomposition. Moreover, the diagonal entries of $\widehat\mxSigma_1$ may be negative, violating the concept of a singular value. These limitations warrant a different approach.

We now describe an alternative method that provides better approximations of the singular values and singular vectors. With the Kronecker sum decomposition (\ref{equation:krondecomposition}), we again begin with the SVD of the first term, $\mxA_1 \kron \mxB_1 = \mxU_1 \mxSigma_1 \mxV_1^T$. Then we rewrite $\mxK$ as
\begin{align*}
\mxK &= \sum\limits_{i=1}^R \mxA_i \kron \mxB_i\\
& = \mxU_1 \mxSigma_1 \mxV_1^T + \sum\limits_{i=2}^R \mxA_i \kron \mxB_i\\
& = \mxU_1\left(\mxSigma_1 + \mxU_1^T\left(\sum\limits_{i=2}^R \mxA_i \kron B_i\right)\mxV_1 \right) \mxV_1^T \\
& = \mxU_1 \left(\mxSigma_1 + \mxW \right) \mxV_1^T\,.
\end{align*}

We wish to compute approximations of the $k$ largest singular values and corresponding
singular vectors of $\bo{K}$. Because $\mx{K}_1 = \mxA_1 \otimes \mxB_1$ is the most significant of the summation terms, we want to use the most significant singular values of $\bo{\Sigma}_1$ in this computation; incorrectly estimating or omitting the largest singular values of $\bo{\Sigma}_1$ causes significant error in the final computed TSVD approximation. The singular values of $\bo{\Sigma}_A$ and $\bo{\Sigma}_B$ are sorted, but when their Kronecker product is taken to get $\bo{\Sigma}_1$, the result is in a sawblade-like (not monotonic) ordering. So simple truncation of $\mx{\Sigma}_1$ (directly or by truncating $\mx{\Sigma}_A$ and $\mx{\Sigma}_B$) does not produce the $k$ most significant entries of $\mx{\Sigma}_1$.

The most straightforward way around this is to reorder the singular values into sorted order. If $\mx{P}$ is the permutation that reorders $\mx{\Sigma}_1$ so that $\mx{P}^T \mx{\Sigma}_1 \mx{P}$ has its diagonal sorted in descending magnitude, we can re-write $$\mx{K}_1 = \mx{U}_1 \mx{P} (\mx{P}^T \mx{\Sigma}_1 \mx{P}) \mx{P}^T \mx{V}_1^T.$$

From there,
$$\mx{K} = \mx{U}_1 \mx{P} [\mx{P}^T (\mx{\Sigma}_1 + \mx{W}) \mx{P}] \mx{P}^T \mx{V}_1^T.$$
Then let $\sigma_{1,j}$ denote the $j$\textsuperscript{th} singular value in $\mxSigma_1$, and define the matrix
$
\mxSigma_{1,k} = \text{diag}(\sigma_{1,1}, \hdots, \sigma_{1,k}) 
$
containing the largest $k$ singular values of $\mxA_1 \kron \mxB_1$, i.e., the first $k$ entries of $\mx{P}^T \mx{\Sigma}_1 \mx{P}$. Similarly, define diagonal matrix 
$
\widehat{\mxSigma}_{0} = \mbox{diag}(\sigma_{1,k+1}, \ldots, \sigma_{1,N})
$
containing the remaining smallest singular values of $\mxA_1 \kron \mxB_1$ (the notation will be clear within the following derivation). Define $\mx{U}_1 \mx{P} = \bar{\mx{U}}_1$ and $\mx{V}_1 \mx{P} = \bar{\mx{V}}_1$, and partition
$\mx{P}^T \mx{W} \mx{P} =
\left[\begin{array}{cc}
  \mx{W}_{11} & \mx{W}_{12}\\
  \mx{W}_{21} & \mx{W}_{22}
\end{array}\right]$. With this notation, we can write $\mxK$ as
\begin{eqnarray*}
\mxK & = &
  \bar{\mx{U}}_1
  \left(
    \left[ \begin{array}{cc}
      \mxSigma_{1,k} & 0 \\
      0 & \widehat{\mxSigma}_0
    \end{array} \right]
    +
    \left[ \begin{array}{cc}
      \mxW_{11} & \mxW_{12} \\
      \mxW_{21} & \mxW_{22}
    \end{array} \right]
  \right)
  \bar{\mx{V}}_1^T \\
& = &
  \bar{\mx{U}}_1
  \left(
    \left[ \begin{array}{cc}
      \mxSigma_{1,k} + \mxW_{11} & 0 \\
      0 & \widehat{\mxSigma}_0
    \end{array} \right]
    +
    \left[ \begin{array}{cc}
      0 & \mxW_{12} \\
      \mxW_{21} & \mxW_{22}
    \end{array} \right]
  \right)
  \bar{\mx{V}}_1^T \\
& = &
  \bar{\mx{U}}_1
  \left(
    \left[ \begin{array}{cc}
      \widehat{\mx{U}}_t \widehat{\mxSigma}_t \widehat{\mx{V}}_t^T & 0 \\
      0 & \widehat{\mx{\Sigma}}_0
    \end{array} \right]
    +
    \left[ \begin{array}{cc}
      0 & \mxW_{12} \\
      \mxW_{21} & \mxW_{22}
    \end{array} \right]
  \right)
  \bar{\mx{V}}_1^T \\
& = &
  \bar{\mx{U}}_1
  \left[ \begin{array}{cc}
    \widehat{\mx{U}}_t & 0 \\
    0 &\mxI
  \end{array} \right]
  \left(
    \left[ \begin{array}{cc}
      \widehat{\mxSigma}_t & 0 \\
      0 & \widehat{\mxSigma}_0
    \end{array} \right]
    +
    \left[ \begin{array}{cc}
      0 & \widehat{\mx{U}}_t^T\mxW_{12} \\
      \mxW_{21} \widehat{\mx{V}}_t & \mxW_{22}
    \end{array} \right]
  \right)
  \left[ \begin{array}{cc}
    \widehat{\mx{V}}_t^T & 0 \\
    0 & \mxI
  \end{array} \right]
  \bar{\mx{V}}_1^T \\
& = &
  \widehat{\mx{U}}
  \left(
    \left[ \begin{array}{cc}
      \widehat{\mxSigma}_t & 0 \\
      0 & \widehat{\mx{\Sigma}}_0
    \end{array} \right]
    +
    \left[ \begin{array}{cc}
      0 & \widehat\mxW_{12} \\
      \widehat\mxW_{21} & \mxW_{22}
    \end{array} \right]
  \right)
  \widehat{\mx{V}}^T 
\end{eqnarray*}
where $\widehat{\mxU}_t \widehat{\mxSigma}_t \widehat{\mxV}_t^T$ is the SVD of the $k \times k$ matrix $\mxT = \mxSigma_{1,k} + \mxW_{11}$, and we define $\widehat\mxW_{12} = \widehat{\mxU}_t^T\mxW_{12}$ and $\widehat\mxW_{21} = \mxW_{21}\widehat{\mx{V}}_t$. If $k = N$ (the size of $\mxK$), then we have the full SVD of $\mxK$.  When $N$ is large, computing the full SVD of $\mxK$ is impractical, but if $k$ is of modest size (e.g., on the order of 1000) then it is feasible to compute the SVD of $\mxT$.

We have
\begin{equation}
\label{eq:modifiedSVD}
\mxK = \widehat\mxU \left( \left[ \begin{array}{cc} \widehat\mxSigma_t & 0 \\ 0 & \widehat\mxSigma_{0}\end{array} \right]
  + \left[ \begin{array}{cc} 0 & \widehat\mxW_{12} \\ \widehat\mxW_{21} & \mxW_{22} \end{array} \right]
  \right) \widehat\mxV^T\,.
\end{equation}
From this, let $\widehat{\mx{\Sigma}}_k = \widehat{\mx{\Sigma}}_t$, $\widehat\mxU_k$ be the a matrix containing the first $k$ columns of $\widehat\mxU$, and 
$\widehat\mxV_k$ be the matrix containing the first $k$ columns of $\widehat\mxV$. Using this notation, we can form a truncated SVD approximation
of $\mxK$ as
\begin{equation}
\label{eq:approximateTSVDoperator}
  \mx{K}_{\mbox{\footnotesize} TSVD} \approx \widehat{\mx{U}}_k \widehat{\mx{\Sigma}}_k \widehat{\mx{V}}_k^T \, .
\end{equation}
Note that we can also expedite computation of $\mxW$ by using a truncated Kronecker sum decomposition, $\mx{K} \approx \sum\limits_{i=1}^r \mx{A}_i \kron \mx{B}_i$ (recall that we use $R$ to denote the full Kronecker rank of $\mxK$, and $r \leq R$ to denote an approximate Kronecker rank).

In actual implementations, the matrices $\widehat\mxU_k$ and $\widehat\mxV_k$ are not formed explicitly. Instead, the components are stored individually. For example, $\mx{U}_A$ and $\mx{U}_B$  are stored to represent $\mx{U}_1$, the reordering map is kept rather than the full permutation matrix $\mx{P}$, and $\mxU_t$ is stored explicitly due to its small size. Then, multiplications are computed as a sequence of operations (multiplications, reordering, and truncation). For example, to compute the product $\widehat{\mx{V}}^T\bo{d}$,
\begin{enumerate}
\item compute $(\mx{V}_A^T \kron \mx{V}_B^T)\bo{d}$ using Kronecker properties,
\item permute the result using the the mapping representing left multiplication with $\mx{P}^T$,
\item truncate to $k$ rows, and then
\item left multiply the result by $\mx{V}_t^T$.
\end{enumerate}

Multiplication with $\widehat{\mx{U}}_k^T$ follows the same pattern, and multiplications with $\widehat{\mx{V}}_k$ and $\widehat{\mx{U}}_k$ follow the pattern in reverse with the transpose of truncation being padding with zeros. Using this approach enables storage even as $N$ gets large: the storage cost is $O(N + k^2) = O(n^2 + k^2)$, whereas storage of the full matrices $\widehat{\mx{U}}_k$ and $\widehat{\mx{V}}_k$ requires $O(Nk)$ space. Because $k \ll N$, the difference is large.

To summarize, we have developed an algorithm that leverages the benefits of Kronecker products to decompose the matrix $\mxK$ into an approximated SVD.  The algorithm uses two levels of approximation: the number of terms used in the Kronecker sum decomposition of $\mxK$, and the truncation index $k$ used to determine the size of $\mxT = \mxU_t \mxSigma_t \mxV_t^T$. Leveraging these approximations results in considerable time and storage savings.

\subsection{Time Complexity Analysis}
\label{ssec:complexity}

In this subsection we show that our method is computationally efficient, with an $O(n^3r + k^2r + k^3)$ running time for our applications. Although this is slower than the baseline method \cite{nagy2003kronecker, kamm1998kronecker}, which runs in $O(n^3r)$ on the same applications, the method is still computationally feasible for moderate choices of $k$. 

The computational cost for each step of the truncated SVD algorithm is as follows:

\begin{itemize}
\item The first step of computing the Kronecker sum decomposition (\ref{equation:krondecomposition}) is critical, but the time complexity depends on the structure and sparsity of $\mx{K}$.  For imaging applications considered in this paper, the cost is at most $O(n^3)$, where it is assumed images have $n \times n$ pixels, and $\mxK$ is an $N \times N$ matrix, with $N = n^2$.  In more general cases where $\mxK$ is sparse, then the cost for computing the Kronecker sum decomposition will depend on the level of sparseness, and on the chosen Kronecker rank, $r$. Without sparsity or structure in $\mx{K}$, the cost is $O(N^3) = O(n^6)$, which is typically infeasible for large $n$.
\item The cost of computing SVDs of $\mxA_1$ and $\mxB_1$ is $O(n^3)$.
\item Sorting the diagonal $\mx{\Sigma}_1$ takes $O(n^2log(n))$  time. This also gives the permutation mapping used in the next step.
\item Recall $\mx{P}^T \mx{W} \mx{P} =
\left[\begin{array}{cc}
  \mx{W}_{11} & \mx{W}_{12}\\
  \mx{W}_{21} & \mx{W}_{22}
\end{array}\right]$ where $\mx{W} = \sum\limits_{i=2}^r\left(\mxU_{A}^T \mx{A}_i \mx{V}_A \kron \mx{U}_B^T \mx{B}_i \mx{V}_B \right)$. Computing $\mx{W}_{11}$ has two main steps per index in the summation: forming the products $\mxU_{A}^T \mx{A}_i \mx{V}_A$ and $\mx{U}_B^T \mx{B}_i \mx{V}_B$, and using the permutation mapping to multiply entries of each product to their correct entry of $\mxW_{11}$. Forming the matrix products takes $O(n^3)$ time because the matrices are size $n \times n$. Applying the permutation mapping and computing the $i^\text{th}$ partial sum of $\mxW_{11}$ takes $O(k^2)$ time. $r-1$ total iterations are calculated, so the total cost is $O(n^3r + k^2r)$.
\item Forming $\mx{T}$ and computing its SVD takes $O(k^3)$ time.
\end{itemize}
The the total cost of computing the approximate truncated SVD is $O(n^3r + k^2r + k^3)$.

\subsection{Approximation Quality}
\label{sec:Theory}

In this section we provide theoretical results bounding the difference in quality of our TSVD operator approximation (\ref{eq:approximateTSVDoperator}) relative to the true TSVD operator (\ref{eq:TSVDoperator}). We start with bounds on the quality of the computed singular vector subspaces, similar to the results presented by Fierro and Bunch \cite{fierro1995bounding} for the case of URV and ULV factorizations. In their work, they suggested such bounds as a potential diagnostic measure to assess quality of the approximate subspaces. Following the subspace bounds, we then derive bounds for the errors of the approximate pseudoinverse and TSVD solution.

We begin by deriving a bound for the singular vector subspaces. Using the notation for the true SVD defined in equation (\ref{eq:FullSVD}), the ``signal'' subspace of the inverse problem (\ref{eq:DIP}) is the span of the columns of $\mx{U}_k$ and the ``noise'' subspace is the span of the columns of $\mx{V}_0$. One quality measure for approximate TSVD operators is the distance between the true and approximated signal subspaces and noise subspaces. These distances are measured by \cite{fierro1995bounding, fierro1995accuracy, li2011moore}
$$
  \|\mxU_k^T\widehat\mxU_0\|_2 \quad \mbox{and} \quad 
  \|\mxV_k^T\widehat\mxV_0\|_2\,.
$$

\begin{theorem}
\label{theorem:SubspaceQuality}
Consider the factorizations of $\mxK$ given in equations (\ref{eq:FullSVD}) and (\ref{eq:modifiedSVD}), where $\sigma_i$ denotes a true singular value and $\widehat{\sigma}_i$ denotes an approximate singular value. Then
\begin{equation}
\label{eq:SignalSubspaceBound}
  \|\mx{U}_k^T\widehat{\mx{U}}_0\|
  \leq
  \frac{
    \sigma_k \|\widehat{\mx{W}}_{21}\|
    +
    \|\widehat{\mx{W}}_{12}\| \|\widehat{\mx{\Sigma}}_0+\mx{W}_{22}\|
  }{
    \sigma_k^2 - \|\widehat{\mx{\Sigma}}_0 + \mx{W}_{22}\|^2
  }
\end{equation}
and
\begin{equation}
\label{eq:NoiseSubspaceBound}
  \|\mx{V}_k^T \widehat{\mx{V}}_0\|
  \leq
  \frac{
    \sigma_k \|\widehat{\mx{W}}_{12}\|
    +
    \|\widehat{\mx{W}}_{21}\|\ \|\widehat{\mx{\Sigma}}_0 + \mx{W}_{22}\|
  }{
    \sigma_k^2 - \|\widehat{\mx{\Sigma}}_0 + \mx{W}_{22}\|^2
  }
\end{equation}
\end{theorem}
\begin{proof}
To prove these bounds, first notice that
$$
  \mxK\mxV_k = \mxU_k\mxSigma_k \quad \Rightarrow \quad \mxU_k^T = \mxSigma_k^{-1}\mxV_k^T\mxK^T
$$
and so
\begin{eqnarray}
 \mxU_k^T\widehat\mxU_0 
 & = & 
  \mxSigma_k^{-1}\mxV_k^T
  \left[ \begin{array}{cc}
    \widehat\mxV_k & \widehat\mxV_0
  \end{array} \right]
  \left[ \begin{array}{cc}
    \widehat\mxSigma_k & \widehat\mxW_{21}^T \\
    \widehat\mxW_{12}^T & \widehat\mxSigma_0 + \mxW_{22}^T
  \end{array} \right]
  \left[ \begin{array}{c}
    \widehat\mxU_k^T \\ \widehat\mxU_0^T
  \end{array} \right]
  \widehat\mxU_0^T 
  \nonumber \\ 
 & = & 
  \mxSigma_k^{-1} \mxV_k^T
  \left[ \begin{array}{cc}
    \widehat\mxV_k & \widehat\mxV_0
  \end{array} \right]
  \left[ \begin{array}{c}
    \widehat\mxW_{21}^T \\
    \widehat\mxSigma_0 + \mxW_{22}^T
  \end{array} \right]  
  \nonumber \\ 
 & = & 
  \mxSigma_k^{-1}
  \left(
    \mxV_k^T\widehat\mxV_k \widehat\mxW_{21}^T
    + 
    \mxV_k^T\widehat\mxV_0
    \left(
      \widehat\mxSigma_0 + \mxW_{22}^T
    \right)
  \right)\,. 
\label{eq:UkU0}
\end{eqnarray}
Similarly, observe that
\begin{eqnarray}
 \mxV_k^T \widehat\mxV_0 
 & = & 
  \mx{\Sigma}_k^{-1} \mx{U}_k^T \mx{K} \widehat{\mx{V}}_0 
  \nonumber \\
 & = & 
  \mxSigma_k^{-1} \mx{U}_k^T \widehat{\mx{U}}_k \widehat{\mx{W}}_{12}
  +
  \mx{\Sigma}_k^{-1} \mx{U}_k^T \widehat{\mx{U}}_0
  \left(
    \widehat{\mx{\Sigma}}_0 + \mx{W}_{22}
  \right)\,.
  \label{eq:VkV0}
\end{eqnarray}
Substituting the above relation for $\mxV_k^T\widehat\mxV_0$ into (\ref{eq:UkU0}) we obtain
$$
  \mxU_k^T\widehat\mxU_0
  =
  \mxSigma_k^{-1} \mxV_k^T \widehat\mxV_k \widehat\mxW_{21}^T
  + 
  \mxSigma_k^{-2}
  \left(
    \mxU_k^T \widehat\mxU_k \widehat\mxW_{12}
    +
    \mxU_k^T \widehat\mxU_0
    \left(
      \widehat\mxSigma_0 + \mxW_{22}
    \right)
  \right)
  \left(
    \widehat\mxSigma_0 + \mxW_{22}^T
  \right)\,.
$$
Taking norms,
$$
  \|\mxU_k^T\widehat\mxU_0\|
  \leq
  \frac{1}{\sigma_k} \|\widehat\mxW_{21}\|
  +
  \frac{1}{\sigma_k^2}
  \left(
    \|\widehat\mxW_{12}\| \|\widehat\mxSigma_0 + \mxW_{22}\|
    +
    \|\mxU_k^T \widehat\mxU_0\| \|\widehat\mxSigma_0 + \mxW_{22}\|^2
  \right).
$$
After algebraic manipulation we obtain the bound for the signal subspace:
$$
  \|\mxU_k^T\widehat\mxU_0\|
  \leq
  \frac{
    \sigma_k \|\widehat\mxW_{21}\|
    +
    \|\widehat\mxW_{12}\| \|\widehat\mxSigma_0 + \mxW_{22}\|
  }{
    \sigma_k^2 - \|\widehat\mxSigma_0+\mxW_{22}\|^2
  }\,.
$$
The bound for the noise subspace is proved similarly.  By substituting the relation (\ref{eq:UkU0}) into (\ref{eq:VkV0}) and taking norms, we obtain:
$$
  \|\mxV_k^T \widehat\mxV_0\|
  \leq
  \frac{
    \sigma_k \|\widehat\mxW_{12}\|
    +
    \|\widehat\mxW_{21}\| \|\widehat\mxSigma_0 + \mxW_{22}\|
  }{
    \sigma_k^2 - \|\widehat\mxSigma_0 + \mxW_{22}\|^2
  }\,.
$$
\end{proof}

Next, we develop the bound for the relative error of the pseudoinverse of the TSVD operator approximation  given in (\ref{eq:approximateTSVDoperator}).

\begin{theorem}
\label{theorem:PseudoinverseQuality}
Consider the true TSVD operator $\mxK_\text{TSVD}$ defined in (\ref{eq:TSVDoperator}) and its approximation $\widehat\mxK_\text{TSVD}$ given by (\ref{eq:approximateTSVDoperator}). Define $\varphi = \left(1+\sqrt{5}\right)/2$. Then

\begin{equation}
\label{eq:PseudoinverseBound}
  \frac{
    \left\Vert
      \mxK_\text{TSVD}^\dagger - \widehat\mxK_\text{TSVD}^\dagger
    \right\Vert
  }{
    \left\Vert
      \mxK_\text{TSVD}^\dagger
    \right\Vert
  }
  \leq
  \frac{\varphi}{\widehat\sigma_k}
  \left(
    \sigma_1 \|\mxV_k^T \widehat\mxV_0\|
    +
    \|\widehat\mxW_{21}\|
  \right).
\end{equation}
\end{theorem}

\begin{proof}
The proof starts with a perturbation result for pseudoinverses presented in \cite{wedin1973perturbation}: If ${\bm C}$ is an acute perturbation \cite{li2011moore} of ${\bm D}$, with ${\bm D} = {\bm C} +  \delta{\bm C}$, then 
\begin{equation}\label{eq:WedinPseudo}
\| {\bm C}^\dagger -  {\bm D}^\dagger\|\leq \varphi \| {\bm C}^\dagger \|\|{\bm D}^\dagger\|\| \delta{\bm C}\|.
\end{equation}
Since
\begin{eqnarray}\nonumber
 \|\mxK_\text{TSVD} - \widehat\mxK_\text{TSVD}\| 
 & = & 
  \| \mxK\mxV_k\mxV_k^T - 
    (
      \mxK\widehat\mxV_k\widehat\mxV_k^T
      -
      \widehat\mxU_0\widehat\mxW_{21}\widehat\mxV_k^T
    )
  \| \\
 & \leq &
  \|\mxK\| \|\mxV_k\mxV_k^T - \widehat\mxV_k \widehat\mxV_k^T\|
  +
  \|\widehat\mxW_{21}\|,
  \nonumber
\end{eqnarray}
and it is proved in \cite{golub2013matrix} that $\|\mxV_k\mxV_k^T - \widehat\mxV_k\widehat\mxV_k^T\| = \|\mxV_k^T\widehat\mxV_0\|$, we get
\begin{equation}\nonumber
  \|\mxK_\text{TSVD} - \widehat\mxK_\text{TSVD}\|
  \leq
  \|\mxK\| \|\mxV_k^T\widehat\mxV_0\| + \|\widehat\mxW_{21}\|.
\end{equation}
$\widehat\mxK_\text{TSVD}$ is an acute perturbation of $\mxK_\text{TSVD}$. So by (\ref{eq:WedinPseudo}),
\begin{equation}\label{eq:absErrorPseudo}
  \left\Vert
    \mxK_\text{TSVD}^\dagger - \widehat\mxK_\text{TSVD}^\dagger
  \right\Vert
  \leq
  \varphi
  \|\mxK_\text{TSVD}^\dagger\|
  \|\widehat\mxK_\text{TSVD}^\dagger \| 
  \left(
    \|\mxK\|\|\mxV_k^T \widehat\mxV_0\| + \|\widehat\mxW_{21}\|
  \right).
\end{equation}
Dividing both sides of the inequality by $\|\mxK_\text{TSVD}^\dagger\|$, we obtain
\begin{eqnarray}\nonumber
  \frac{
    \left\Vert
      \mxK_\text{TSVD}^\dagger - \widehat\mxK_\text{TSVD}^\dagger
    \right\Vert
  }{
    \left\Vert \mxK_\text{TSVD}^\dagger \right\Vert
  }
  & \leq &
  \varphi 
  \|\widehat\mxK_\text{TSVD}^\dagger \| 
  \left(
    \|\mxK\| \|\mxV_k^T \widehat\mxV_0\| + \|\widehat\mxW_{21}\|
  \right) \\
  & = & \nonumber
  \frac{\varphi}{\widehat\sigma_k}
  \left(
    \sigma_1 \|\mxV_k^T \widehat\mxV_0\| + \|\widehat\mxW_{21}\|
  \right).
\end{eqnarray}
\end{proof}

This theorem tells us that a good approximation of the pseudoinverse requires two conditions: a small distance between the ``noise'' subspaces, and a small ratio of the true largest singular value $\sigma_1$ to the approximated $k^{th}$ singular value $\widehat\sigma_k$. 

The bound for the approximated solution, computed using $\widehat{\mx{K}}^\dagger_\text{TSVD}$, shares these requirements. The following theorem bounds the relative error in the approximate TSVD solution $\bo{x}_\text{TSVD}$.

\begin{theorem}
Consider the TSVD filtered solution $\bfx_\text{TSVD}$ defined in (\ref{eq:TSVDSol}) and let the approximate TSVD solution $\widehat\bfx_\text{TSVD}$ be given by $\widehat\bfx_\text{TSVD} = \widehat\mxK_\text{TSVD}^\dagger \bfd$. Further, define the residual $\bfr = \bfd - \mxK\bfx_\text{TSVD}$. Then
\begin{equation}\label{theorem:SolutionQuality}
  \frac{
    \left\Vert \bo{x}_\text{TSVD} - \widehat{\bo{x}}_\text{TSVD} \right\Vert
  }{
    \left\Vert \bo{x}_\text{TSVD} \right\Vert
  } 
  \leq
  \frac{
    \varphi \sigma_1
  }{
    \sigma_k \widehat{\sigma}_k
    \sqrt{
      1 - \frac{\| \bo{r} \|^2}{\| \bo{d} \|^2}
    }
  }
  \left(
    \sigma_1 \|\mxV_k^T \widehat{\mx{V}}_0\| + \|\widehat{\mx{W}}_{21}\|
  \right).
\end{equation}
\end{theorem}

\begin{proof}
From inequality (\ref{eq:absErrorPseudo}), we obtain
\begin{eqnarray}
  \left\Vert \bfx_\text{TSVD} - \widehat\bfx_\text{TSVD} \right\Vert 
  & \leq &
  \nonumber
  \left\Vert
    \mxK_\text{TSVD}^\dagger - \widehat\mxK_\text{TSVD}^\dagger
  \right\Vert
  \left\Vert \bfd \right\Vert \\
  & \leq & \nonumber
  \varphi
  \|\mxK_\text{TSVD}^\dagger\|
  \|\widehat\mxK_\text{TSVD}^\dagger \| 
  \left(
    \|\mxK\|\|\mxV_k^T \widehat\mxV_0\| + \|\widehat\mxW_{21}\|
  \right)
  \|\bfd\|.
\end{eqnarray}
By the triangle inequality and submultiplicativity of induced norms,
\begin{equation}\nonumber
  \|\bfx_\text{TSVD}\|
  \geq
  \frac{\|\bfd\|}{\|\mxK\|}
  \sqrt{1-\frac{\| \bfr \|^2}{\| \bfd \|^2}},
\end{equation}
it follows that 
\begin{eqnarray}
  \frac{
    \left\Vert \bfx_\text{TSVD} - \widehat\bfx_\text{TSVD} \right\Vert
  }{
    \left\Vert \bfx_\text{TSVD}\right\Vert
  }
  & \leq & \nonumber
  \frac{\varphi}{\sqrt{1-\frac{\| \bfr \|^2}{\| \bfd \|^2}}}
  \|\mxK_\text{TSVD}^\dagger\|
  \|\widehat\mxK_\text{TSVD}^\dagger \| 
  \|\mxK\|
  \left(
    \|\mxK\| \|\mxV_k^T \widehat\mxV_0\| + \|\widehat\mxW_{21}\|
  \right) \\
  & = & \nonumber
  \frac{
    \varphi \sigma_1
  }{
    \sigma_k \widehat{\sigma}_k \sqrt{1-\frac{\| \bfr \|^2}{\| \bfd \|^2}}
  }
  \left(
    \sigma_1 \|\mxV_k^T \widehat\mxV_0\| + \|\widehat\mxW_{21}\|
  \right).
\end{eqnarray}
\end{proof}

There are two limitations of note for the four bounds provided here. All of the bounds are derived in exact arithmetic, but computations are performed in floating point arithmetic. Therefore, the true error may exceed the bound due to numerical errors; an example of this behavior is seen in Section~\ref{ssec:bounds}. Additionally, if the approximated singular values have high accuracy and the ratio $\frac{\sigma_1}{\sigma_k}$ is large, then the pseudoinverse and true solution bounds are extremely pessimistic. We reiterate that those two bounds are most useful for matrices with highly clustered singular values.

\section{Experimental Results}\label{sec:experiments}
In this section, we show the strengths and limitations of our algorithm in various experimental settings. We show that, in practice, this method is more accurate but slower than the baseline (Section~\ref{ssec:performance}, Section~\ref{ssec:precondition}) and less accurate but faster than standard TSVD software (Section~\ref{ssec:comparison_tsvd}). We also show the strengths and limitations of our derived bounds (Section~\ref{ssec:bounds}). All experiments are written and executed in MATLAB.

\subsection{Comparison to Standard Truncated SVD Software}\label{ssec:comparison_tsvd}
Experimental evidence suggests that our algorithm is faster, although less accurate, than standard software for computing a truncated singular value decomposition of a matrix. Specifically, we compare our proposed algorithm with the PROPACK \cite{larsen2004propack} truncated SVD routine, \texttt{lansvd}, which is based on Lanczos (Golub-Kahan) bidiagonalization. The \texttt{lansvd} routine is open source but the randomized method \cite{halko2011finding} routines are not; we compare only to the Lanczos routine to ensure the presented results are for properly optimized implementations of the algorithms.

The \texttt{lansvd} routine computes the top $k$ singular values and vectors of a matrix. It does not require the explicit formation of the input matrix, but instead allows the user to pass function handles that compute the action of the matrix on a vector. These functions are necessary when memory constraints prohibit explicit formation of the matrix, and they can also be faster than full, explicit multiplications for certain sparse matrix structures. Because our problem has such an exploitable structure, our experimental comparison uses the function call form of \texttt{lansvd}.

The main work in the Lanczos bidiagonalization routine is construction of a sequence of orthonormal vectors that bidiagonalize a matrix. That is, for a given matrix $\mxK \in \mathds{R}^{N\times N}$, the Lanczos bidiagonalization procedure iteratively constructs the columns of matrices $\mxU \in \mathds{R}^{N \times N}$ and $\mxV \in \mathds{R}^{N \times N+1}$ such that $\mxU^T \mxK \mxV$ is bidiagonal \cite{larsen2004propack}. Because of imprecision incurred through division by small numbers, the columns of $\mxU$ and $\mxV$ do not remain orthonormal throughout the iterations \cite{golub2013matrix}. Periodic reorthonormalization prevents these instabilities from accumulating excessive error. This requires extra computation but considerably improves the accuracy of the method. The \texttt{lansvd} routine allows users to adjust the level of reorthogonalization, as well as an overall stopping criteria, through optional parameters.

A timing test demonstrated that our algorithm is faster than the \texttt{lansvd} routine, even with its reorthogonalization effectively turned off. Both algorithms, as well as a full SVD computation with MATLAB's \texttt{svd} function, were run on a 4096$\times$4096 blur operator matrix originating from a 64$\times$64 speckle PSF from astronomical imaging; see the test problem {\tt PRblurspeckle} in \cite{gazzola2017irtools}. We tested the Kronecker product decomposition algorithm on a variety of choices of Kronecker rank, but report here the results for the full Kronecker rank of $r = 64$, which were the slowest runs of our algorithm. Additionally, we tested for multiple accuracies of Lanczos bidiagonalization by tuning the parameters for the method. The choices yielding the fastest (with reorthogonalization turned off and error bound checking effectively turned off) and most accurate (using a default of high accuracy and low error tolerance) runs are reported. All times shown are the average of 5 runs of each algorithm for each singular value rank tested. Our algorithm ran quicker than the \texttt{lansvd} routine for all truncation sizes tested, as shown in \autoref{fig:lanczos_timings}.

\begin{figure}[H]
\centering
\includegraphics[scale=.35]{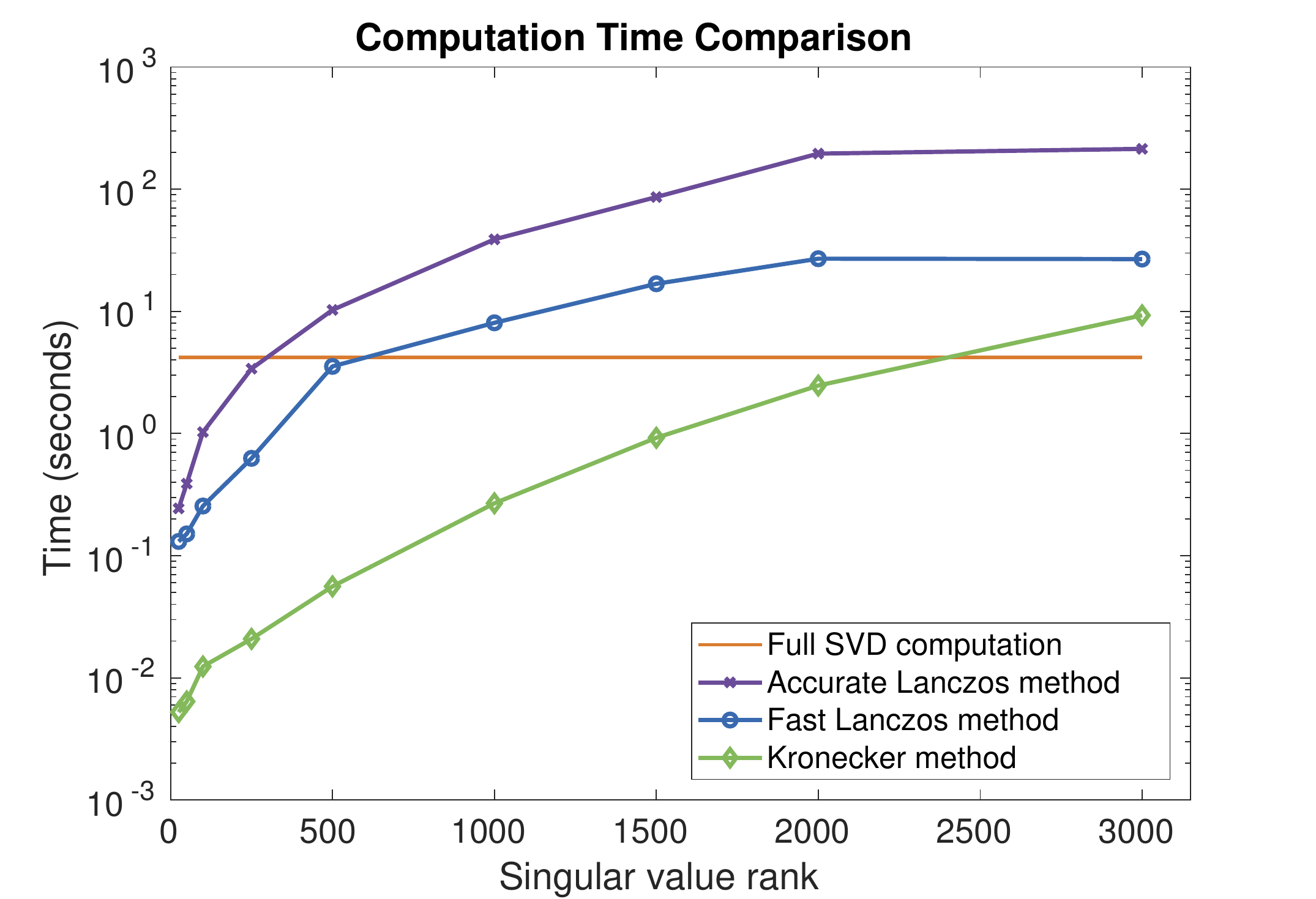}
\caption[Comparison of Lanczos and Kronecker timings]{A comparison of Lanczos bidiagonalization and Kronecker product decomposition SVD methods. The Kronecker method ran faster than both the accurate and fast Lanczos methods. Time for computing a full SVD of the matrix is included.}
\label{fig:lanczos_timings}
\end{figure}

Both methods eventually became slower than computing a full SVD. This is unsurprising. The matrix could fit in memory for this example, and both the \texttt{lansvd} and Kronecker-based methods require overhead that is needlessly costly for very small matrices. The strengths of these methods are their ability to operate when the full matrix cannot fit in memory. Using them for small matrices is inadvisable.

When run to high accuracy, the PROPACK \texttt{lansvd} routine was much more accurate than our approximated TSVD for this test problem. For each truncation size tested, the relative error in singular values compared to MATLAB's SVD was near zero ($10^{-15}$ to $10^{-16}$, near machine epsilon) for the PROPACK routine on the most significant singular values, while our approximation method had significant relative error on the order of $10^{-6}$. This is expected: our routine is intended as a quick approximation. With reorthogonalization effectively turned off, the error of Lanczos bidiagonalization became extreme. As the number of singular values and vectors computed increased, the relative error reached the order of $10^{10}$.

For applications in which accuracy can be sacrificed for speed, our method outperforms the PROPACK \texttt{lansvd} method.

\subsection{Performance} \label{ssec:performance}
Our proposed TSVD approximation successfully solves image deconvolution problems for which a direct solution is feasible. We compare the performance of our method with the baseline method. The results depend on the decay of singular values in the rearrangement matrix $\tilde{\mx{K}}$. For matrices with sharp decay in singular values, the baseline and proposed reordering methods perform similarly and comparably to a ground truth TSVD solution. When instead the singular value decay is slow, the proposed method strongly outperforms the baseline. We demonstrate performance for image deconvolution problems with each category of matrix.

We begin showing the performance on a problem for which all methods perform similarly. To enable comparison with ground truth, we used a $64 \times 64$ test image of a satellite, which can be obtained from \cite{nagy2004iterative}. The PSF, obtained from \texttt{AtmosphericBlur30.mat} in the same package and shown in \autoref{fig:speckle_PSF}, was of the same size. We used zero boundary conditions to construct the blur operator. The Kronecker summation was not truncated ($r = R$), but the matrix rank was truncated to 600 singular values. $2\%$ Gaussian noise was added to the blurred image. Because truncation in the SVD alone did not produce a sufficiently smooth solution, we additionally used Tikhonov regularization \cite{tikhonov1977solutions} with regularization parameter $\lambda = 0.001$.

\begin{figure}[H]
\includegraphics[scale=.35]{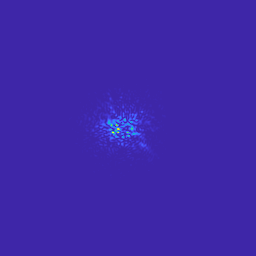}
\caption{The PSF used the first example in \sslink{performance}.}
\label{fig:speckle_PSF}
\end{figure}

The true image, blurred noisy image, and restored image are shown in \autoref{fig:SatelliteResults}. The restoration using MATLAB's \texttt{svds} method and the baseline method also shown for comparison.

\begin{figure}[H]
\centering
\begin{tabular}{c c}
True & Blurred\\
\includegraphics[scale=1.3]{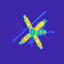} &
\includegraphics[scale=1.3]{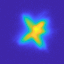}\\
&
\end{tabular}

\begin{tabular}{c c c}
MATLAB \texttt{svds} & Proposed TSVD & Baseline TSVD\\
\includegraphics[scale=1.3]{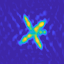} &
\includegraphics[scale=1.3]{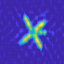} &
\includegraphics[scale=1.3]{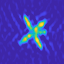}
\end{tabular}
\caption[Restoration Results]{A comparison of the restoration using our proposed method to baseline method and the true TSVD computed with \texttt{svds}. All methods produce visually similar results.}
\label{fig:SatelliteResults}
\end{figure}

How good of an approximation are our computed singular values? Quite good, especially for the largest singular values. The relative error in the singular values are shown in \autoref{fig:RelErrorComparison} for our method and the baseline.

\begin{figure}[H]
\centering
\includegraphics[scale=.5]{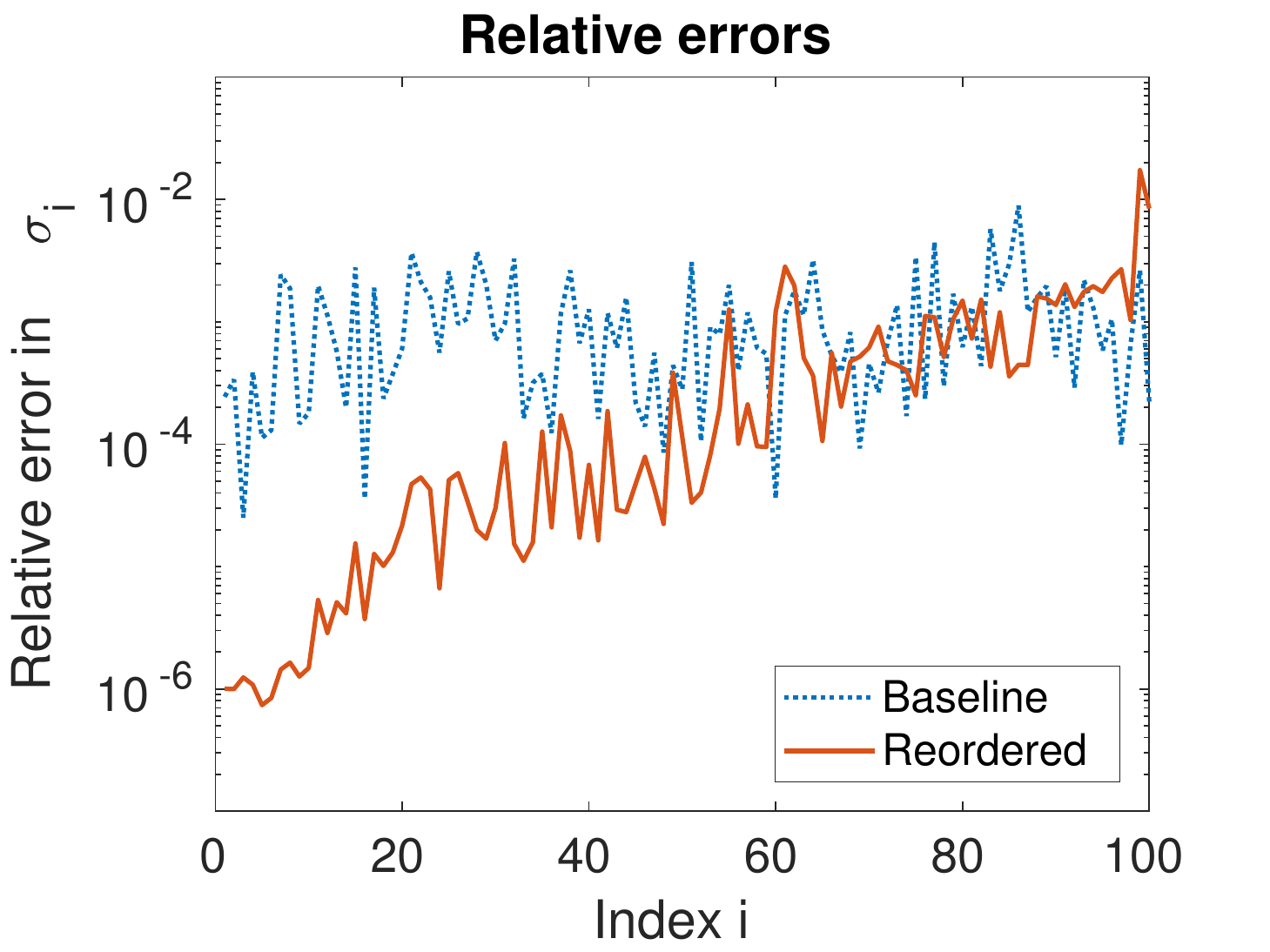}
\caption{Relative errors of the computed singular values. The proposed reordering method has higher accuracy for the largest singular values than the baseline method.}
\label{fig:RelErrorComparison}
\end{figure}

Our singular value approximation naturally has higher relative (and absolute) error closer to the truncation boundary. Because the method is fast, an option for practical use is to choose a larger than necessary truncation index $k$, compute the approximate TSVD, and truncate down further to remove the least accurate values. In general, the error is lower than the baseline method for the largest singular values.

In addition to the previous test, we tested on a larger problem with a PSF representing motion blur, shown in \autoref{fig:motion_PSF}. The singular values of the corresponding blur operator $\tilde{\mx{K}}$ have a much slower decay for this problem than the previous example. We tested using a $256 \times 256$ version of the satellite test image with a severe motion blur PSF from \cite{gazzola2017irtools}. Again we used zero boundary conditions, and the blurred image was contaminated with $1\%$ Gaussian noise. This problem is less ill-posed than the previous example, so no Tikhonov regularization was added.

\begin{figure}[H]
\includegraphics[scale=.4]{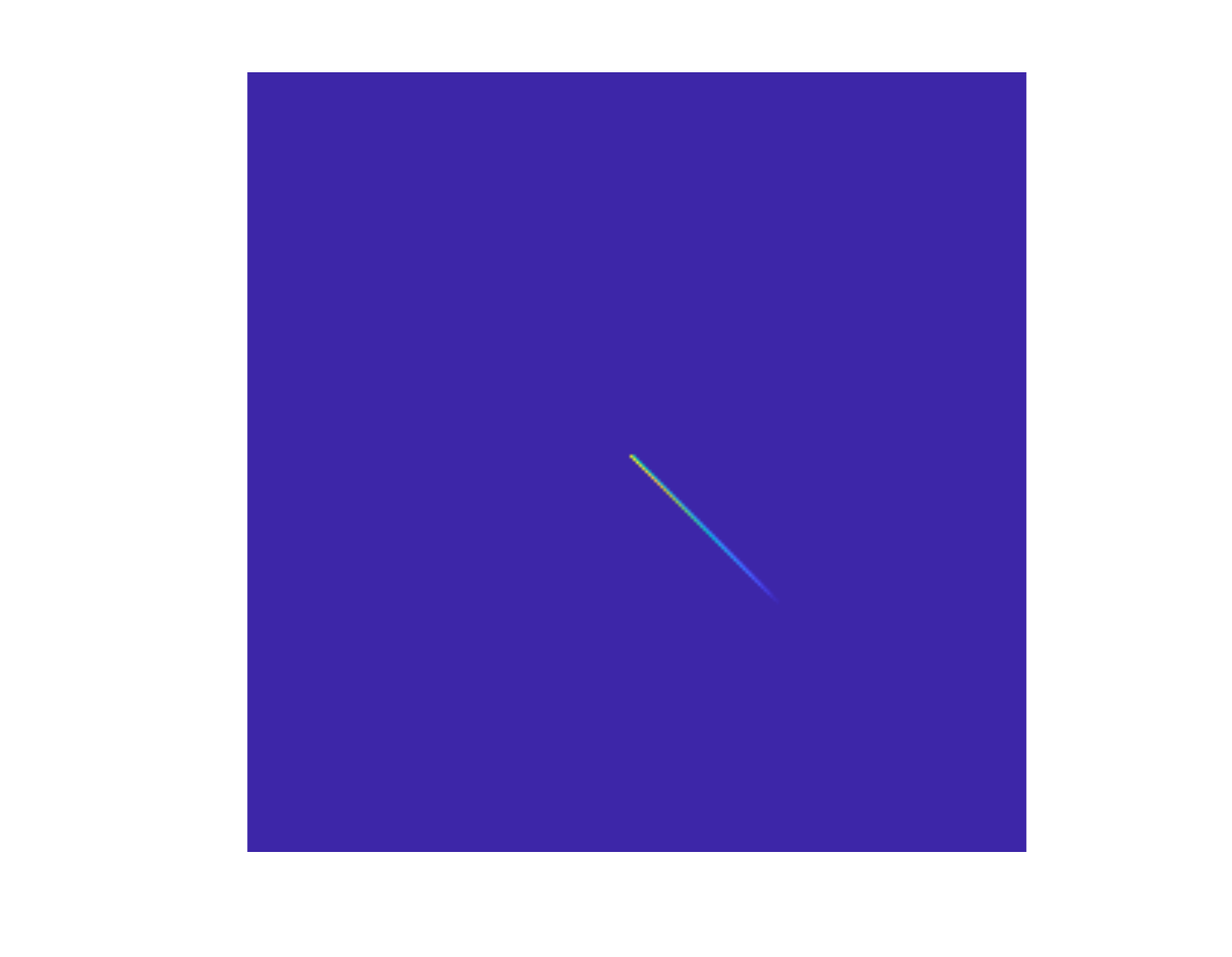}
\caption{The motion blur PSF used in the second example of  \sslink{performance}.}
\label{fig:motion_PSF}
\end{figure}

\begin{figure}[H]
\centering
\begin{tabular}{c c}
True & Blurred\\
\includegraphics[scale=.4]{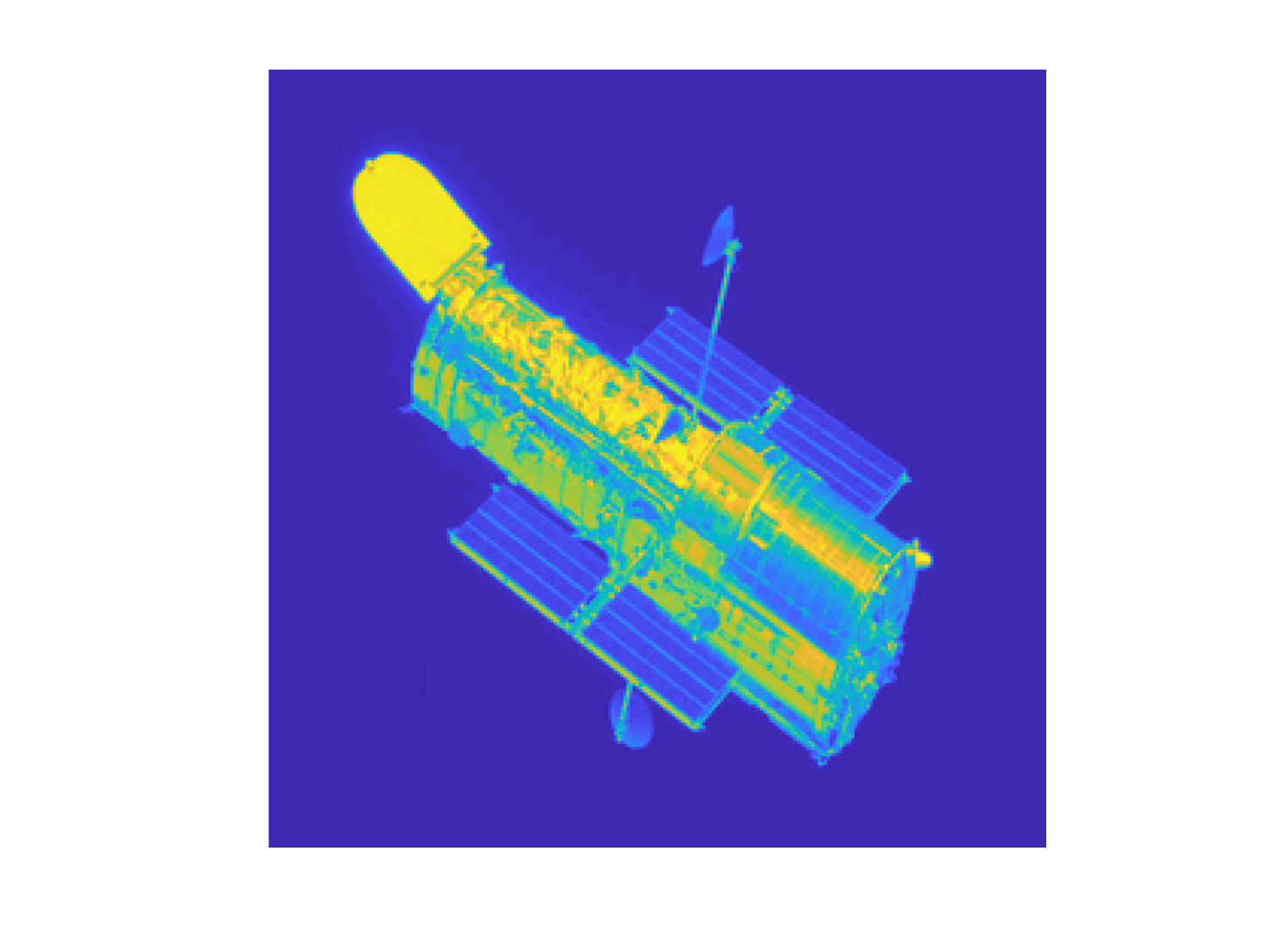} &
\includegraphics[scale=.4]{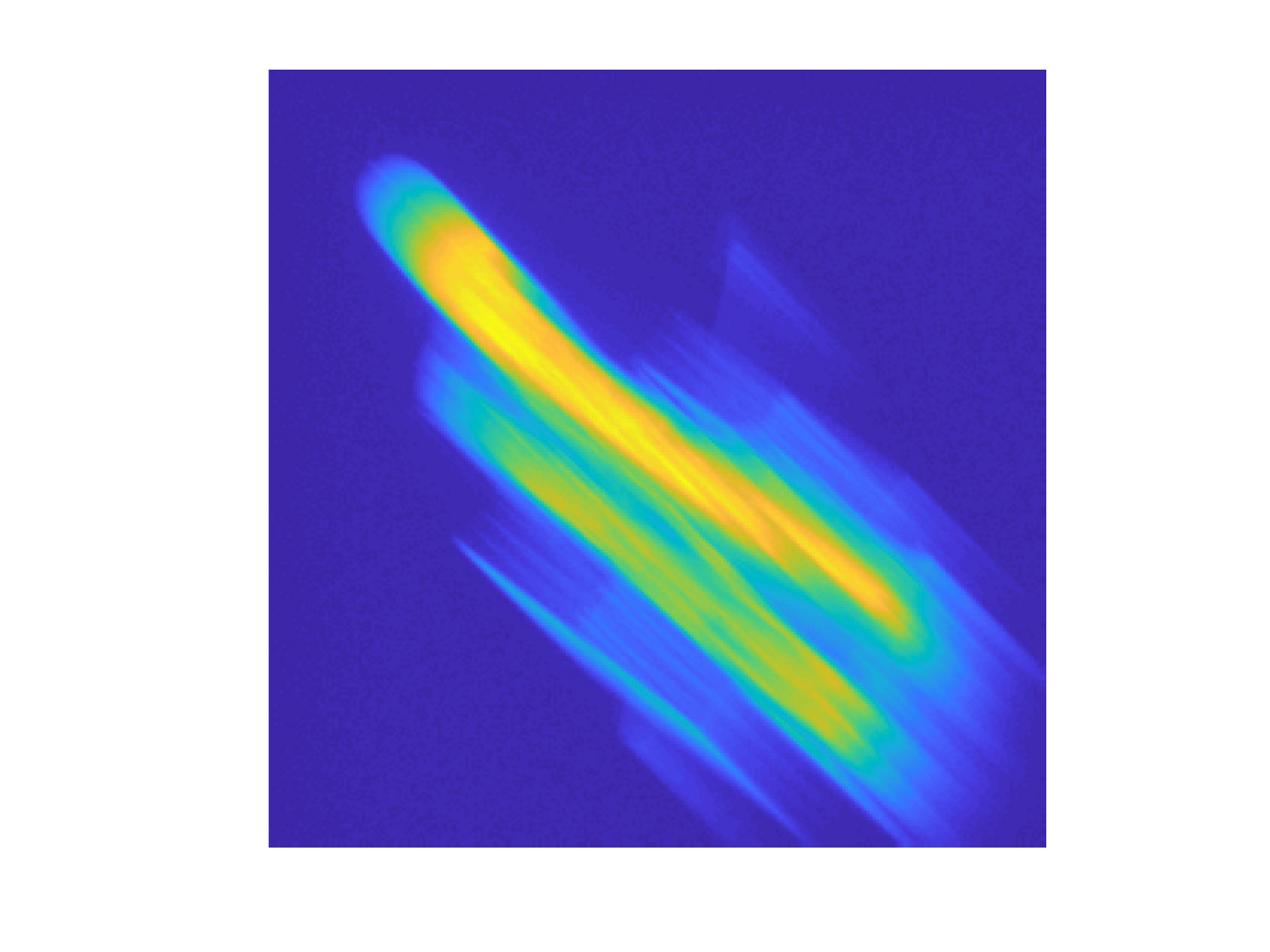}\\
&\\
Proposed TSVD & Baseline TSVD\\
\includegraphics[scale=.4]{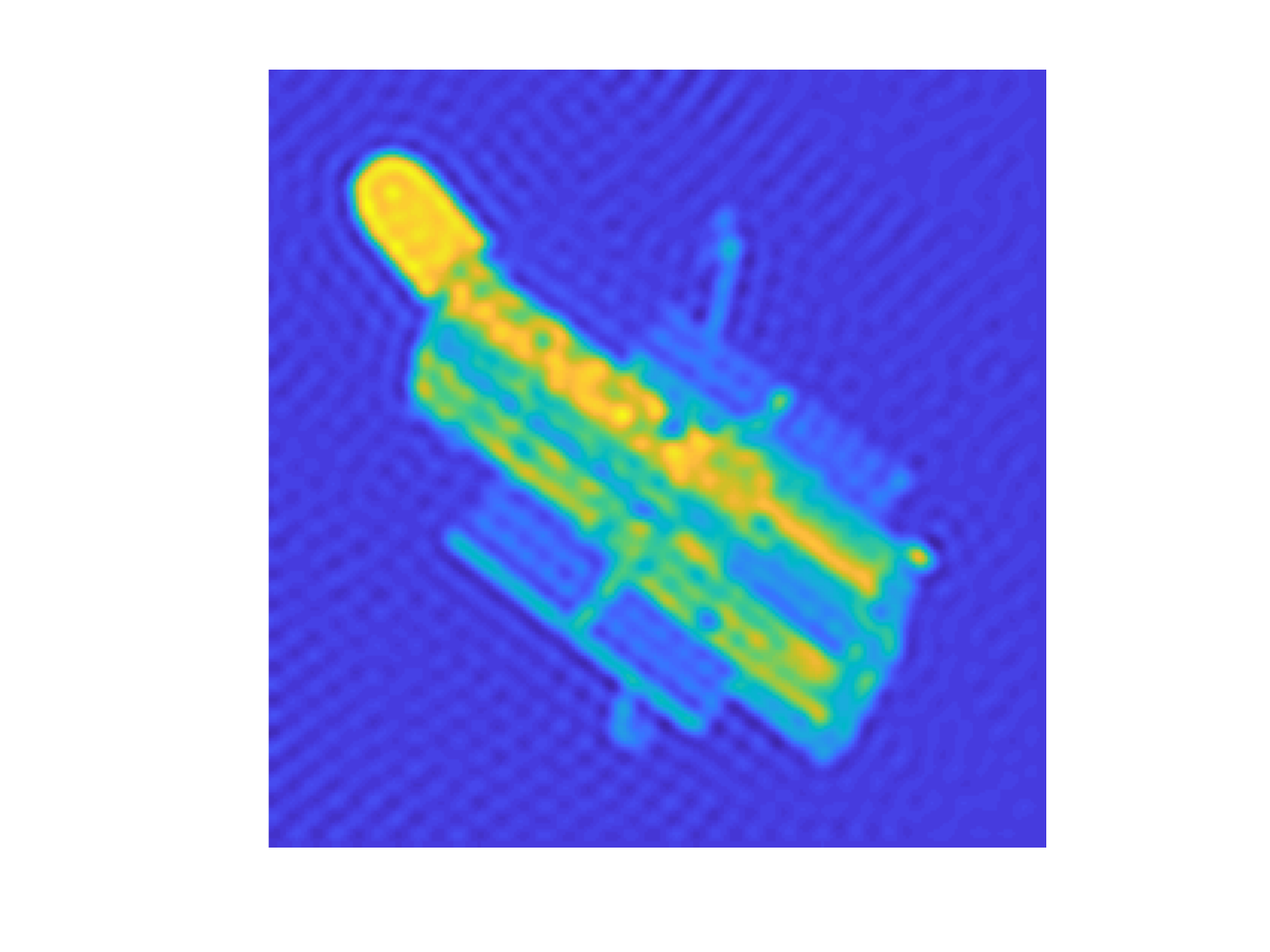} &
\includegraphics[scale=.4]{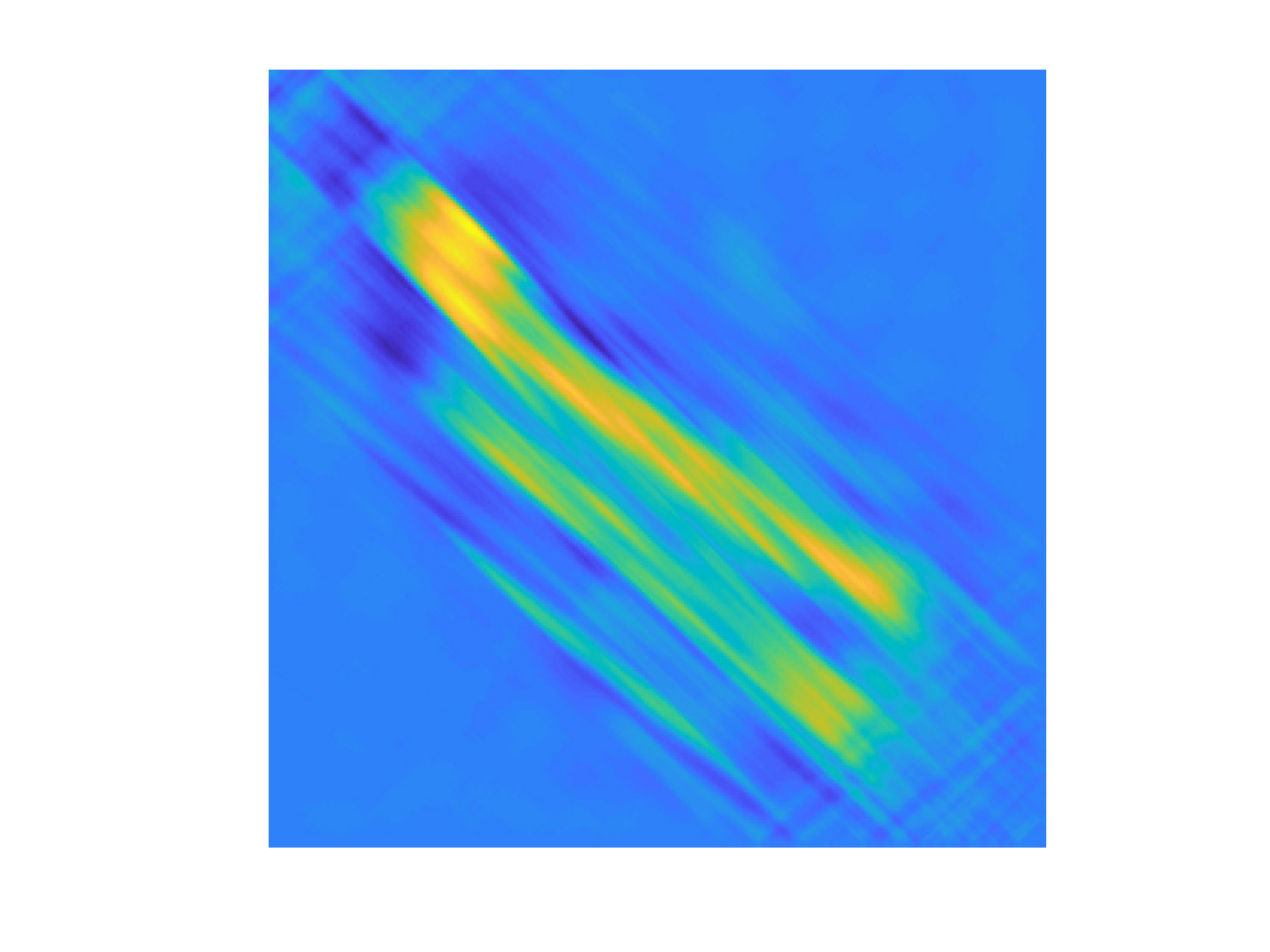}
\end{tabular}
\caption[Restoration Results]{A comparison of the restoration using our proposed method to baseline method for the motion blur example. In this case, our method produces visually superior results.}
\label{fig:motion_results}
\end{figure}

Computing ground truth singular values and vectors was infeasible due to the problem size (here, $\mx{K}$ is size $65563 \times 65563$). However, as is shown in \autoref{fig:motion_results}, the baseline method clearly gives a worse restoration than the reordering method, as expected: when the decay of singular values in $\tilde{\mx{K}}$ is slow, the baseline method generally discards too much data for an accurate approximation. Our proposed reordering method preserves much more information in comparison. For problems like this where the PSF (or, equivalently, $\tilde{\mx{K}}$) has slowly decaying singular values, our method produces results superior to the baseline method.

\subsection{Bounds} \label{ssec:bounds}
In Section~\ref{sec:method}, we provided bounds on the distances between the values computed by our approximate TSVD and the exact TSVD for their signal and noise subspaces, the pseudoinverse, and $\bo{x}_\text{TSVD}$. Here, we show the performance of these bounds on an actual test problem and discuss their limitations.

The problem setup for this section is identical to the experiments in Section~\ref{ssec:performance}, except that we use a different PSF originating from atmopsheric blur (see \texttt{AtmosphericBlur50.mat} from  \cite{nagy2004iterative}). For each bound, we demonstrate the effect of computing to a truncation index $k = 100$ then truncating down to a smaller effective rank. The bounds for the signal subspace are shown in \autoref{fig:subspace_bound}.

\begin{figure}[H]
\centering
\includegraphics[scale=.45]{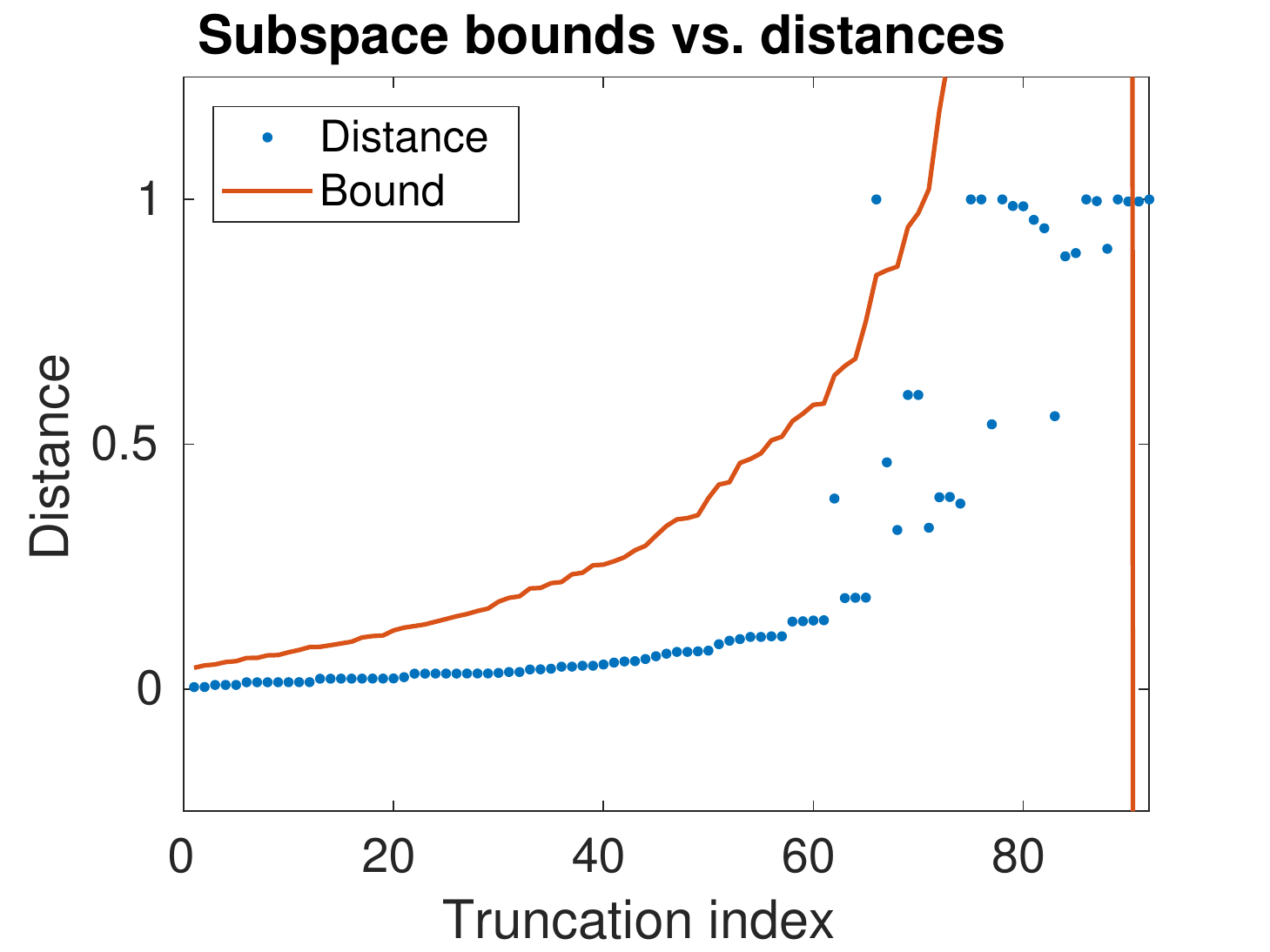}
\caption[Subspace bounds]{The distance between true and computed signal subspaces versus bound on the same quantity. In this example, the bound is useful until a truncation index of around 70, but then the limitations of the bound render it unusable. The noise subspace bounds are visually indistinguishable for this test problem.}
\label{fig:subspace_bound}
\end{figure}

For this test, numerical errors cause one point to have a higher actual distance than the bound. As mentioned previously, this is because the bound is derived in exact arithmetic, but the actual distances are calculated in floating point precision. Further, the distances are guaranteed to be in the range $[0, 1]$, but the theoretical bound is not. The bound first increases and exceeds this range as, referring to the notation in (\ref{eq:SignalSubspaceBound}), $k$ increases and $\sigma_k$ therefore decreases, with $\sigma_k$ approaching $\|\widehat{\mx{\Sigma}}_0 + \mx{W}_{22}\|$. However, eventually $\sigma_k$ becomes smaller than  $\|\widehat{\mx{\Sigma}}_0 + \mx{W}_{22}\|$ due to error in the approximated singular values, and the bound becomes negative, indicating that the distance is at least $0$ (which is known a priori). The plot is truncated at this point.

For the pseudoinverse and solution bounds we use the true subspace distances, rather than the derived bounds (\ref{eq:SignalSubspaceBound}) or (\ref{eq:NoiseSubspaceBound}). This enables the bound to be usable past the point at which the bound in \autoref{fig:subspace_bound} becomes negative. Despite this, the bounds are extremely loose this example. See \autoref{fig:pseudo_and_x_bounds} for an illustration of this.

\begin{figure}[H]
\begin{center}
\begin{tabular}{cc}
\includegraphics[scale=.45]{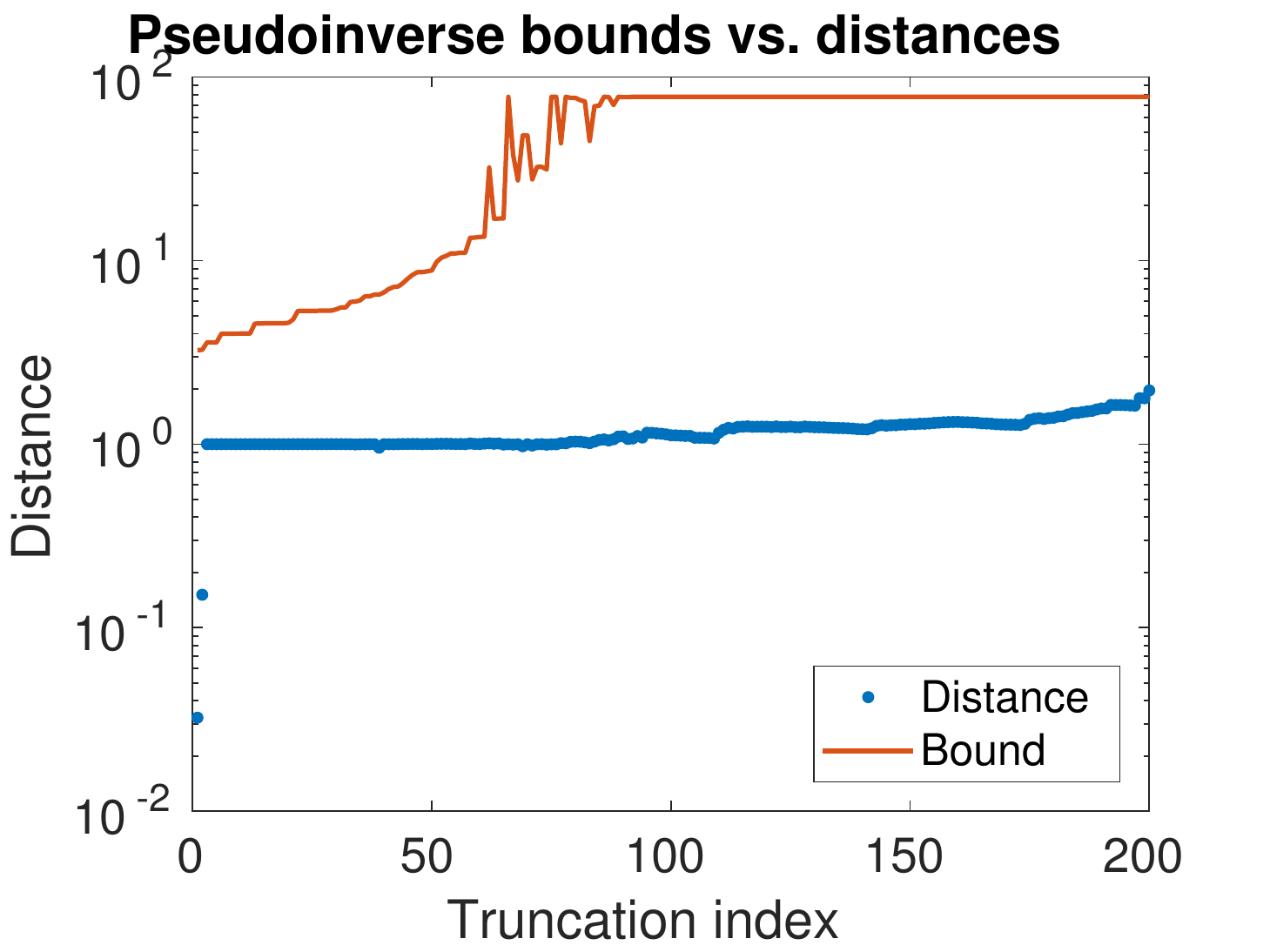} &
\includegraphics[scale=.45]{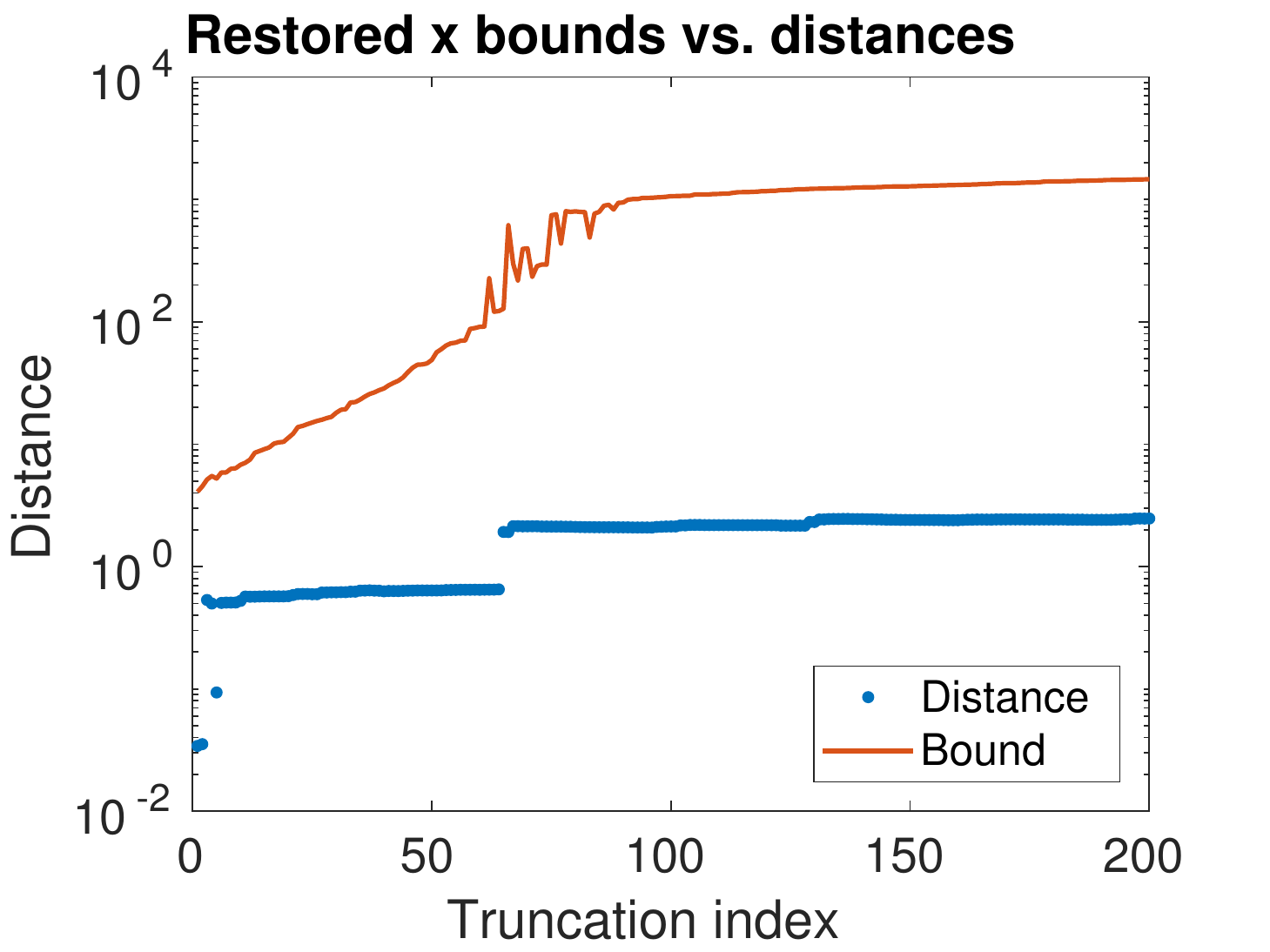}
\end{tabular}
\end{center}
  {\caption{This plot on the left shows bounds for the pseudo inverse, and the plot on the right shows bounds for the TSVD solution.}\label{fig:pseudo_and_x_bounds}}
\end{figure}

These bounds are too loose to be useful. This is because the linear problem is ill-posed, with fairly quickly decaying singular values of the matrix $\mx{K}$. Therefore the ratios $\frac{\sigma_1}{\sigma_k}$ and $\frac{\sigma_1}{\widehat{\sigma}_k}$ are large, loosening the bounds. We caution against using these bounds without a priori knowledge of the singular value clustering of the matrix $\mx{K}$.

\subsection{Preconditioning} \label{ssec:precondition}
In this subsection, we show that both our proposed method and the baseline method are effective preconditioners, saving wall-clock time compared to unpreconditioned systems. Further, we discuss how to choose between the baseline and proposed methods for preconditioning.

We tested the baseline method and our reordering method as preconditioners for a preconditioned conjugate gradient least squares (PCGLS) formulation of the image deconvolution problem and compare the time taken for each method to converge. The PSF originates from an astronomical imaging problem (\texttt{satellite.mat} in \cite{nagy2004iterative}) with image size 256$\times$256; the operator $\mxK$ cannot fit in memory. We ran the PCGLS algorithm no preconditioner, with the baseline method as a preconditioner, and finally with a preconditioner constructed from our new reordering method, with rank $k = 1500$. We used Tikhonov regularization with regularization parameter $\lambda = 0.001$, which prevented the baseline method from incurring undue noise. Each time reported is the average of 20 trials. The times are shown in \autoref{fig:pcg_timings} below.

\begin{figure}[H]
\centering
\includegraphics[scale=.5]{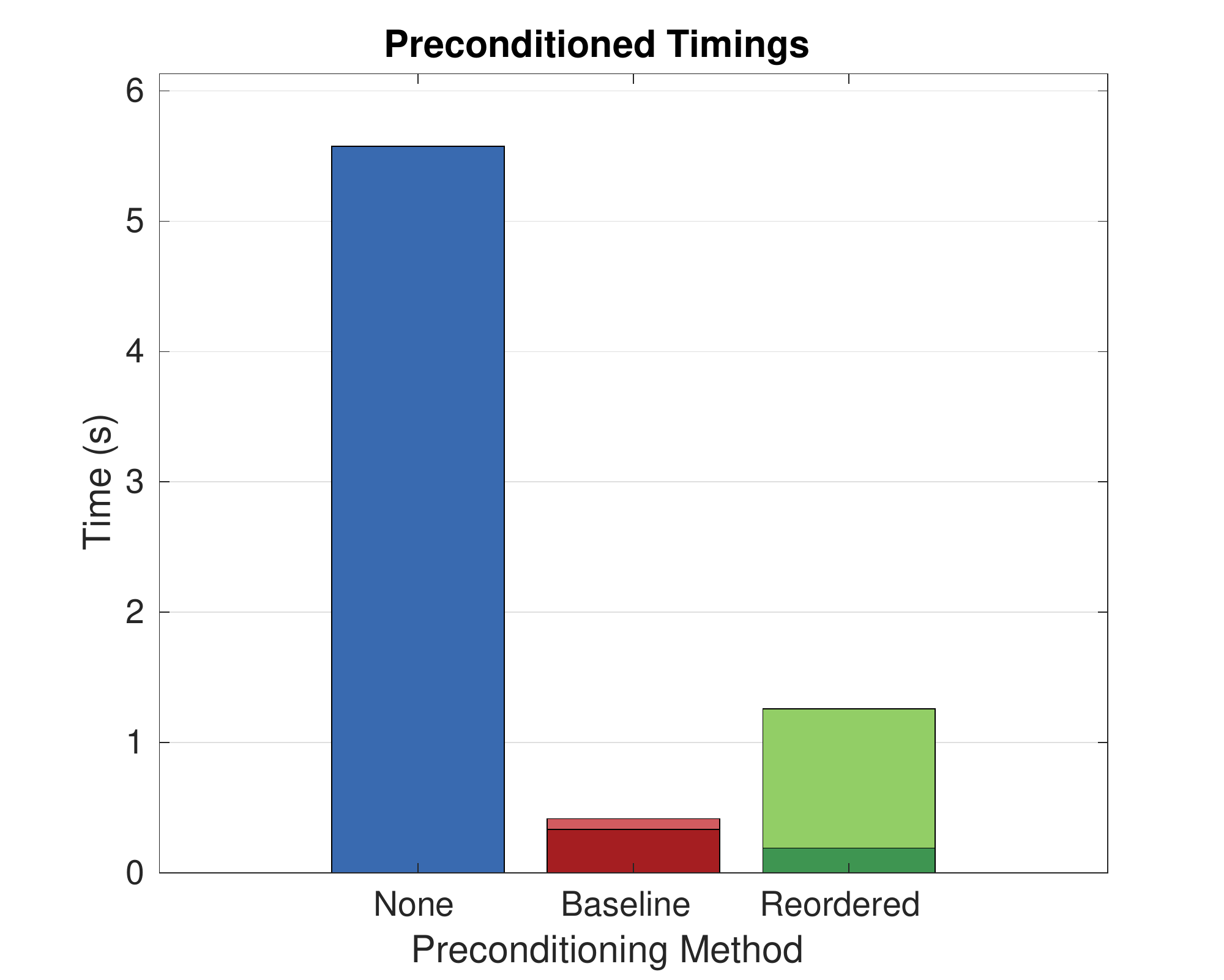}
\caption{Comparative times for preconditioned solutions of an image restoration problem. Solving using conjugate gradient least squares method without a preconditioner takes the most time of all methods tested, and the baseline was fastest. For both preconditioners, the time taken to compute the preconditioner is shown in a light color at the top of the bar, and the time taken to run the PCGLS method with that preconditioner is shown in a darker color. The reordered method was run with rank $k = 1500$.}
\label{fig:pcg_timings}
\end{figure}

For overall speed, the baseline method is by far the best. Total computation time averaged 0.42 seconds seconds compared to the unpreconditioned 5.6 seconds. However, the computation time is heavily skewed; the PCGLS algorithm takes more iterations with the less accurate baseline preconditioner than it takes with the more accurate reordering method preconditioner (16 iterations versus 7). The result is that the time taken for PCGLS iterations is less for the reordering method (.19 seconds) than for the baseline method (.34 seconds). For this problem, if there were multiple right-hand sides $\bo{b}$, the reordering method would eventually produce time savings over the baseline.

Both our currently-proposed method and the baseline method are effective preconditioners. For problems with a single right-hand side, the baseline method is an appropriate choice, but with many right-hand sides the reordering method is preferred.

\section{Related Work}\label{sec:related_works}
In this section we detail the context of this work among related works. First, we detail the prior work based on Kronecker product decompositions in Section~\ref{ssec:related_kron}. For the remaining subsections, we discuss related works that use alternative approaches to perform truncated SVD approximations. Although these final related works do not exploit Kronecker structure, it is possible for them to be reformulated to do so. Efficiently implementing the algorithms to use Kronecker structure is non-trivial, but could provide a direction for future research.

\subsection{Kronecker Decompositions}\label{ssec:related_kron}
The original idea of performing efficient Kronecker product decompositions on blur operator matrices came from Nagy in 1996 \cite{nagy1996decomposition}. This became the basis for a method to compute an approximated SVD by Kamm and Nagy \cite{kamm1998kronecker}. The same authors later provided theoretical justification for the method \cite{kamm2000optimal}. The original algorithm applied only to block Toeplitz matrices with Toeplitz blocks corresponding to zero or periodic boundary conditions and only to 2D problems. Additionally, the original algorithm uses a simple computation for the singular value matrix $\mxSigma$ as detailed in Section~\ref{sec:method}, although \cite{kamm2000optimal} uses an additional diagonal weighting matrix in its construction.

Since the publication by Kamm and Nagy \cite{kamm2000optimal}, the algorithm has been extended in three general ways. First, the original 2D algorithm was modified to enable use on 3D problems \cite{nagy2006kronecker}. Second, Kilmer and Nagy extended the work from banded matrices to dense matrices \cite{kilmer2007kronecker}. And third, the choices of boundary conditions have expanded to reflexive \cite{nagy2003kronecker}, anti-reflective \cite{perrone2006kronecker}, and whole-sample symmetric/reflective boundary conditions \cite{lv2012kronecker}, among others. All prior work uses the simple $\mxSigma$ computations proposed in \cite{kamm1998kronecker} and \cite{kamm2000optimal}.

\subsection{Golub-Kahan-Lanczos Bidiagonalization Methods}
There are various alternatives to Kronecker-based methods for computing approximate truncated singular value decompositions of matrices. One alternative is based on Lanczos bidiagonalization. These methods are popular because of their high accuracy and flexibility.

Lanczos bidiagonalization transforms a general matrix into a lower bidiagonal matrix, which serves as a foundation for computing the singular value decomposition. Recall that, because the process is unstable, restarting and reorthogonalization are required to maintain stability \cite{golub2013matrix}. Each Lanczos bidiagonalization step requires two matrix-vector multiplications: one with the matrix $\mxK$, and one with $\mxK^T$.  These multiplications and reorthogonalization account for the bulk of the computation time of the algorithm.

PROPACK is a popular implementation of Lanczos bidiagonalization routines for computing truncated SVDs \cite{larsen2004propack}. The routine \texttt{lansvd} approximately computes an SVD with user-specified options. These options include a truncation index for computing truncated SVDs up to a specified number of terms, an overall convergence tolerance for the method, and tolerances for the levels of orthogonalization used in restarts and after each step.

By default, \texttt{lansvd} has higher accuracy but takes longer than our Kronecker product-based SVD approximation. This accuracy can be relaxed to expedite computation, but the method nonetheless remains slower than our proposed algorithm because \texttt{lansvd} does not exploit Kronecker structure in its computations. See Section \ref{ssec:comparison_tsvd} for a detailed comparison.

\subsection{Randomized Methods}
A second alternative set of methods are randomized methods. Halko, Martinsson, and Tropp proposed a generalized random algorithm for computing various factorizations of matrices, including the truncated SVD \cite{halko2011finding}. One version of this algorithm applies to matrices for which computing a matrix-vector product is fast; this is the case described in Section~\ref{sec:experiments}, as the PSF can be applied with chosen boundary conditions without explicitly forming the full blur operator $A$.

The time complexity of the randomized algorithm for computing a truncated SVD approximation with an $n \times n$ PSF is $O(n^4k + n^2k^2)$, while the complexity of our method is $O(n^3r + k^2r + k^3)$ (see formula 6.2 of \cite{halko2011finding} and Section~\ref{sec:notation} to convert notation). In our experience, $k > n$ to be effective, and $r$ may be chosen small enough so as to be roughly constant. Based on these considerations, our method has a practically faster time complexity.

\subsection{Converting Notation of Randomized Methods}\label{sec:notation}
Halko, Martinsson, and Tropp \cite{halko2011finding} provide time complexity for computing a truncated SVD using their randomized method with function handles for computing matrix-vector and matrix-transpose-vector products. Formula 6.2 in their paper provides this complexity, and complexities for related formulations can be found in Section 6.1. Their notation differs from ours considerably. To facilitate interested readers' understanding, we provide a conversion chart below.

\vspace*{12pt}

\begin{center}
Notation Conversion\\
\vspace*{2mm}
\begin{tabular}{|l|l|l|}
\hline
\multicolumn{1}{|c|}{\textbf{Randomized}} & \multicolumn{1}{c|}{\textbf{Kronecker}} & \multicolumn{1}{c|}{\textbf{Meaning}}\\
\hline
$m$ & $N$ & The number of rows in the matrix $\mx{K}$.\\
\hline
$n$ & $N$ & The number of columns in $\mx{K}$.\\
\hline 
$p$ & N/A & Small ``pad factor'' for randomized methods.\\
\hline
$k$ & $k$ & Truncated SVD rank.\\
\hline 
N/A & $n$ & Number of rows and columns in the image; $\sqrt{N}$.\\
\hline
N/A & $r$ & Kronecker rank of $\mx{K}$.\\
\hline
$q$ & N/A & Number of power iterations for ill-conditioned $\mx{K}$.\\
\hline 
\end{tabular}
\end{center}

\vspace*{12pt}

\section{Concluding Remarks}
\label{sec:conclusions}
In this paper, we propose a new method to compute an approximate truncated singular value decomposition of a matrix using Kronecker product summation decompositions. This method gives more accurate results than previously-explored Kronecker-based methods and remains computationally feasible. We provide bounds on various error measures related to the approximation, but caution that two of the derived bounds are only useful for matrices with tight clustering of singular values. In practice, the method works both well and quickly on a variety of problems. The tests we used are a variety of image deconvolution problems, but the method is applicable for any problem for which computing the Kronecker product summation decomposition is cheap.

\newpage

\bibliographystyle{siam}
\bibliography{references}

\end{document}